\documentclass[twoside, 11pt]{article}
\usepackage[utf8]{inputenc}
\usepackage[T1]{fontenc}
\usepackage[english]{babel}
\usepackage{amsmath, amssymb, amsthm, amstext, amsfonts, a4}
\usepackage{dsfont}
\usepackage[svgnames]{xcolor}
\usepackage{graphics,graphicx}
\usepackage{caption}
\usepackage{subcaption}
\usepackage{enumitem}
\usepackage{comment}
\usepackage{tgpagella}
\usepackage[affil-it]{authblk}
\usepackage{empheq}

\usepackage{hyperref}
\hypersetup{%
colorlinks=true,
linkcolor=Navy,
citecolor=Brown,
urlcolor=Navy
}

\usepackage{geometry}
\numberwithin{equation}{section}

\theoremstyle{plain}
	\newtheorem{theorem}{Theorem}[section] 
	\newtheorem{proposition}[theorem]{Proposition}       
	\newtheorem{lemma}[theorem]{Lemma}
	\newtheorem{corollary}[theorem]{Corollary}
        
\theoremstyle{definition}
	\newtheorem{definition}[theorem]{Definition}
	\newtheorem{remark}{Remark}[section]
    \newtheorem{example}{Example}[section]

\theoremstyle{remark}
    \newtheorem*{thx}{Acknowledgments}

\renewenvironment{proof}{\smallskip\noindent\emph{\textbf{Proof.}}%
  \hspace{1pt}}{\hspace{-5pt}{\nobreak\quad\nobreak\hfill\nobreak%
    $\square$\vspace{2pt}\par}\smallskip\goodbreak}

\newenvironment{proofof}[1]{\smallskip\noindent\emph{\textbf{Proof~of~#1.}}%
  \hspace{1pt}}{\hspace{-5pt}{\nobreak\quad\nobreak\hfill\nobreak%
    $\square$\vspace{2pt}\par}\smallskip\goodbreak}
\newcommand{\N}{\mathbb{N}} 
\newcommand{\Z}{\mathbb{Z}} 
\newcommand{\R}{\mathbb{R}} 
\newcommand{\ds}[1]{\displaystyle{#1}}
\newcommand{\limit}[2]{{\ \underset{#1 \to #2}{\longrightarrow} \ }}
\newcommand{\1}{\mathbf{1}} 
\renewcommand{\d}[1]{\mathinner{\mathrm{d}{#1}}} 
\newcommand{\p}{\partial} 
\newcommand{\eps}{\mathrm{\varepsilon}}
\newcommand{\sgn}{\mathop{\rm sgn}}
\newcommand{\abs}[1]{{\left|#1\right|}}
\newcommand{\norm}[1]{{\left\|#1\right\|}}
\newcommand{\open}[2]{\mathopen] #1, #2 \mathclose[}

\newcommand{\Czero}{\mathbf{C}} 
\newcommand{\Ck}[1]{\mathbf{C}^{#1}} 
\newcommand{\Cc}[1]{\mathbf{C}_\mathbf{c}^{#1}}
\newcommand{\Lip}{\mathbf{Lip}} 
\renewcommand{\L}[1]{\mathbf{L}^{#1}} 
\newcommand{\Lloc}[1]{\mathbf{L}_{\mathbf{loc}}^{#1}} 
\newcommand{\W}[2]{\mathbf{W}^{#1, #2}} 
\newcommand{\Wloc}[2]{\mathbf{W}_{\mathbf{loc}}^{#1, #2}}
\newcommand{\Hloc}[1]{\mathbf{H}_{\mathbf{loc}}^{#1}}
\newcommand{\BV}{\mathbf{BV}} 

\newcommand{\god}{\mathrm{god}}
\newcommand{\fint}{F_{\mathrm{int}}}

\newcommand{\cG}{\mathcal{G}}
\newcommand{\cR}{\mathcal{R}}
\newcommand{\cU}{\mathcal{U}}
\newcommand{\cV}{\mathcal{V}}
\newcommand{\cL}{\mathcal{L}}

\begin{document}

\title{\textbf{Convergence of a Finite Volume Scheme for \\ 
Compactly Heterogeneous Scalar Conservation Laws}}

\author{Abraham Sylla$^1$}

\date{}

\maketitle

\footnotetext[1]{\texttt{abraham.sylla@u-picardie.fr} \\
LAMFA CNRS UMR 7352, Université de Picardie Jules Verne \\
33 rue Saint-Leu, 80039 Amiens (France) \\
ORCID number: 0000-0003-1784-4878}

\thispagestyle{empty}

\begin{abstract}
    We build a finite volume scheme for the scalar conservation law 
    $\p_t u + \p_x (H(x, u)) = 0$ with initial condition $u_o \in \L{\infty}(\R, \R)$ 
    for a wide class of flux function $H$, convex with respect to the second variable. 
    The main idea for the construction of the scheme is to use the theory of 
    discontinuous flux. We prove that the resulting approximating sequence converges in 
    $\Lloc{p}(\open{0}{+\infty} \times \R, \R)$, for all 
    $p \in [1, +\infty \mathclose[$, to the entropy solution. 
\end{abstract}






\section{Introduction}

Consider the following Cauchy problem for $x$-dependent scalar conservation law:
\begin{empheq}[left=\empheqlbrace]{align}
    \label{eq:CL}
        \p_t u(t, x) + \p_x \left(H(x, u(t, x)) \right) & = 0 
        \qquad (t, x) \in \open{0}{+\infty} \times \R \nonumber \\
        u(0, x) & = u_o(x) \qquad x \in \R.
    \tag{\textbf{CL}}
\end{empheq}

Equations of this type often occur, for instance, in the modeling of physical 
phenomena related to traffic flow \cite{Whitham2014, Mochon1987, HR1995}, porous 
media \cite{GR1992, Jaffre1995} or sedimentation problems \cite{Diehl1995, Diehl1996}.

It is known that solutions of \eqref{eq:CL} are discontinuous, regardless of the 
regularity of the data. They are to be understood in the entropic sense of 
\cite{Kruzhkov1970}. The following quantity often recurs below,
\begin{equation}
    \label{eq:EntropyFlux}
    \forall x, u, k \in \R, \quad \Phi(x, u, k) \coloneq 
    \sgn(u - k) \; (H(x, u) - H(x, k)).  
\end{equation}  

\begin{definition}
    \label{def:EntropySolution}
    Let $H \in \Ck{1}(\R^2, \R)$ and $u_o \in \L{\infty}(\R, \R)$. We say that 
    $u \in \L{\infty}(\open{0}{+\infty} \times \R, \R)$ is an entropy 
    solution to \eqref{eq:CL} if for all test functions 
    $\varphi \in \Cc{\infty}(\R^+\times \R, \R^+)$ and $k \in \R$, 
    \begin{empheq}{align}
        \label{eq:ES}
        \int_{0}^{+\infty} \int_{\R} 
        \abs{u-k} \p_t \varphi + \Phi(x, u, k) \p_x \varphi \; \d{x} \d{t} 
        - \int_{0}^{+\infty} \int_{\R} 
        \sgn(u -k) \p_x H(x, k) \varphi \; \d{x} \d{t} \nonumber \\
        + \int_{\R} \abs{u_o(x)-k} \varphi(0, x) \d{x} \geq 0.
    \end{empheq}
\end{definition}

In appearance, Definition \ref{def:EntropySolution} is weaker than the classical 
\cite[Definition 1]{Kruzhkov1970} since it does not require the existence of a 
strong trace at the initial time. It is in particular more manageable to limit 
process. Nevertheless, it can be shown that if $H \in \Ck{3}(\R^2, \R)$, then 
Definition \ref{def:EntropySolution} ensures that each entropy solution admits a 
representative belonging to $\Czero(\R^+, \Lloc{1}(\R, \R))$, 
see \cite[Theorem 2.6]{CPS2023}.  
\bigskip 

In \cite{Kruzhkov1970}, the author proved existence, uniqueness and stability with  
respect to the initial datum for \eqref{eq:CL} in the framework of entropy solutions 
under, among others, the growth assumptions
\begin{equation}
    \label{eq:GrowthKruzhkov}
    \p_u H \in \L{\infty}(\R^2, \R) \quad \text{and} \quad 
    \sup_{(x, u) \in \R^2} \left(-\p_{x u}^2 H(x, u) \right) < +\infty.
\end{equation}

These assumptions highlight the fact that the space dependency of $H$ 
in \eqref{eq:CL} was treated as a source term. 
In his paper, Kruzhkov proved existence using the vanishing viscosity method, that is, 
through a parabolic regularization of \eqref{eq:CL}. The following example motivates 
the necessity to relax \eqref{eq:GrowthKruzhkov} in the 1D scalar case.

\begin{example}
    \label{ex:HeterogeneousLWR}
    Fix positive constants $X, V_1, V_2, R_1, R_2$, and let $\cV, \rho \in \Ck{3}(\R, \open{0}{+\infty})$ be such that $\cV(x) = V_1$, resp. $\rho(x) = R_1$, for $x < -X$ and $\cV(x) = V_2$, resp. $\rho(x) = R_2$, for $x > X$. Define 
    \[
        H(x, u) \coloneq \cV(x) \; u \left(1 - \frac{u}{\rho(x)} \right).
    \] 

    Then with this flow, \eqref{eq:CL} is the so-called “LWR” (Lightill-Whitham, 
    Richards) model \cite{LW1955, Richards1956} for a flow of vehicles described by 
    their density $u$ along a rectilinear road with maximal speed, resp. density, 
    smoothly varying from $V_1$ to $V_2$, resp. from $R_1$ to $R_2$. Clearly, $H$ does 
    not satisfy \eqref{eq:GrowthKruzhkov}. Let us signal however, that since $0$ and $\rho$ are stationary solutions to the equation, this example could treated in Kruzhkov's framework for nonnegative initial data bounded from above by $\rho$.
\end{example}

Recently, the authors of \cite{CPS2023} developed an alternative framework to 
tackle \eqref{eq:CL}, one not requiring the growth assumptions 
\eqref{eq:GrowthKruzhkov} and inspired by Example \ref{ex:HeterogeneousLWR}:
\begin{align}
    \label{eq:smoothness}
    \mathbf{Smoothness:} & \quad H \in \Ck{3}(\R^2, \R). \tag{\textbf{C3}} \\[6pt]
    \label{eq:LocalizedSpace}
    \mathbf{Compact~Heterogeneity:}
    & \quad \exists X > 0, \; \forall (x, u) \in \R^2, \quad 
    \abs{x} \geq X \implies \p_x H(x, u) = 0. \tag{\textbf{CH}} \\[5pt]
    \label{eq:convexity}
    \mathbf{Strong~Convexity:} & \quad \forall x \in \R, \; 
    u \mapsto \p_u H(x, u) \; \text{is an increasing} \tag{\textbf{CVX}} \\[5pt] 
    & \quad \text{$\Ck{1}$-diffeomorphism of $\R$ onto itself.} \nonumber
\end{align}

Assumption \eqref{eq:convexity} ensures that for all $x \in \R$, the mapping 
$u \mapsto H(x, u)$ is strictly convex. Naturally, it can be replaced by the strong 
concavity assumption: for all $x \in \R$, $u \mapsto \p_u H(x, u)$ is a decreasing 
$\Ck{1}$-diffeomorphism of $\R$ onto itself.

The condition \eqref{eq:LocalizedSpace} expresses the compact spatial heterogeneity of 
$H$ and is not, apparently, common in the context of conservation laws. We expect that 
it might be relaxed.

Like we previously mentioned, Assumptions 
\eqref{eq:smoothness}--\eqref{eq:LocalizedSpace}--\eqref{eq:convexity} comprise flows 
that do not fit in the classical Kruzhkov framework and that are relevant from the 
modeling point of view, for instance, the flow of Example \ref{ex:HeterogeneousLWR} 
(in the concave case). Other examples are given by flows of the type
\[
    H(x, u) \coloneq \theta(x) h\bigl(u-\ell(x) \bigr) + g(x),
\]

for suitable assumptions on the functions $\theta, h, \ell$ and $g$. More precisely, $\theta, \ell, g \in \Ck{3}(\R, \R)$ have derivatives with compact support, $\theta > 0$ and $h \in \Ck{3}(\R, \R)$ is such that $h'$ is a $\Ck{1}$-diffeomorphism from $\R$ onto itself. \textit{A priori}, those flows cannot be treated with a truncation argument.

On the other hand, Kruzhkov results apply to general balance 
laws in several space dimensions. Let us recall the following.

\begin{theorem}{\cite[Theorem 2.6 and Theorem 2.18]{CPS2023}}
    \label{th:WellPosedness}
    Assume that $H$ satisfies \eqref{eq:smoothness}--\eqref{eq:LocalizedSpace}--\eqref{eq:convexity}. Then for all 
    $u_o \in \L{\infty}(\R, \R)$, the Cauchy problem \eqref{eq:CL} admits a unique 
    entropy solution $u \in \L{\infty}(\open{0}{+\infty}\times \R,\R)$.

    Moreover, if $v_o \in \L{\infty}(\R, \R)$ and if $v$ is the associated entropy 
    solution, then there exists $L > 0$ such that for all $R > 0$ and $t > 0$,
    \[
        \begin{aligned}
            \int_{\abs{x} \leq R} \abs{u(t, x) - v(t, x)} \d{x} & 
            \leq \int_{\abs{x} \leq R+ Lt} \abs{u_o(x) - v_o(x)} \d{x} 
        \end{aligned}
    \]
\end{theorem}

The proof of existence in \cite{CPS2023}, like in \cite{Kruzhkov1970}, relies on the 
vanishing viscosity method, but there, the authors exploit the correspondence with the 
Hamilton-Jacobi equation (and its parabolic approximation), going back and forth 
between the two frameworks, and gathering information for both equations at each step. 

\begin{remark}
    To be precise, let us mention that in \cite{CPS2023}, the convexity assumption 
    \eqref{eq:convexity} is relaxed to a uniform coercivity assumption coupled
    to a genuine nonlinearity assumption:
    \begin{align}
        \label{eq:coercivity}
        \mathbf{Uniform~Coercivity:} & \quad 
        \forall h \in \R, \; \exists \; \cU_h \in \R, \; 
        \forall (x, u) \in \R^2, \nonumber \\ 
        & \quad \abs{H(x, u)} \leq h \implies \abs{u} \leq \cU_h. \tag{\textbf{UC}} \\[6pt]
        \label{eq:nonlinearity}
        \mathbf{Weak~Genuine~NonLinearity:}
        & \quad \text{for a.e.} \; x \in \R, \; \text{the set} \; 
        \{p \in \R \; : \; \p_{uu}^2 H(x, p) = 0 \} \nonumber \\ 
        & \quad \text{has empty interior.} \tag{\textbf{WGNL}}
    \end{align}

    The strong convexity implies \eqref{eq:nonlinearity}, while \eqref{eq:coercivity} 
    follows from the fact that $H$ admits a Nagumo function (see \cite[Theorem 6.1.2]{CannarsaSinestrariSBook}), meaning that exists a function $\phi \in \Czero(\R^+, \R)$ such that:
    \[
        \forall x, u \in \R, \; H(x, u) \geq \phi(\abs{u}) 
        \quad \text{and} \quad 
        \frac{\phi(r)}{r} \limit{r}{+\infty} +\infty.
    \] 

    For instance, 
    \[
        \forall r \in \R^+, \quad 
        \phi(r) \coloneq \inf_{x \in \R} \bigl( \min\{H(x, r), H(x, -r)\} \bigr)
    \]

    is such a function.

    The convexity was used in \cite{CPSInverseDesign} to characterize, 
    for \eqref{eq:CL}, the attainable set and the set of initial data evolving at a 
    prescribed time into a prescribed profile.
\end{remark}

To construct solutions of \eqref{eq:CL}, an alternative to the vanishing viscosity 
is to build and prove the convergence of finite volume schemes. 
The most recent results include the works of \cite{EGH1995, EGGH1998, Champier1999} 
(in the case $H(x, u) = \theta(x) f(u)$), \cite{CH1999} 
(convergence and error analysis), \cite{CHC2001} (with a source term) or 
\cite{Vovelle2002} (on a bounded domain). In all these works, the flux function enters 
the framework of Kruzhkov by satisfying either \eqref{eq:GrowthKruzhkov}, or the 
stronger requirement $\p_x H \equiv 0$. 

The aim of this paper is to build a finite volume scheme for \eqref{eq:CL} and prove 
its convergence to the entropy solution under the assumptions 
\eqref{eq:smoothness}--\eqref{eq:LocalizedSpace}--\eqref{eq:convexity}. As a byproduct, 
we provide an alternate 
existence result for the Cauchy problem \eqref{eq:CL}, one that does not rely on the 
vanishing viscosity method. However, let us stress that the main result of the paper is 
a convergence result and not an existence result. To the author's knowledge, no 
convergence result is available for flux functions that do not verify 
\eqref{eq:GrowthKruzhkov}. 

In the convergence analysis of any standard finite volume scheme for \eqref{eq:CL}, a 
difficulty to tackle is the obtaining of \textit{a priori} $\L{\infty}$ bounds for the 
approximating sequence. Indeed, because of the space dependency, constants are 
not steady states anymore, therefore the maximum principle breaks down. In 
\cite{EGH1995, EGGH1998, Champier1999, CH1999, CHC2001, Vovelle2002}, this issue is 
overcome either with the requirement $\p_x H \equiv 0$, ensuring that constants are 
steady states, or with Assumption \eqref{eq:GrowthKruzhkov}, which provides a time 
dependent $\L{\infty}$ bounds for the numerical solution. 
We choose a new approach, relying on the theory of 
discontinuous flux, see \cite{AG2003, AJG2004, AGM2005, AKR2011, Diehl1995, GR1992, 
KRT2003, KT2004, Mishra2005, Towers2000} and the references therein. 
For the construction of the scheme, the space dependency of $H$ is discretized, 
and the idea is to treat each edge of the mesh as an interface of a discontinuous flux 
problem. Outside the compact of space dependency, our scheme reduces to a standard 
Godunov method. The contribution of the discontinuous flux theory is that we can build 
by hand discrete steady states of the scheme, making it, in a way well-balanced.

The next section presents our main result.

\section{Description of the scheme and main result}
\label{sec:TheScheme}

In this section, we describe the construction of our scheme for \eqref{eq:CL}. The 
novelty lies in the technique used, where after discretizing the space dependency of 
$H$, at each edge of the mesh, we solve a Riemann problem for a conservation law 
with discontinuous flux.

Fix a spatial mesh size $\Delta x > 0$ and time step $\Delta t > 0$. For all 
$n \in \N$ and $j \in \Z$, set the notations 
\[
    t^n \coloneq n \Delta t, \quad x_j \coloneq j\Delta x, \; 
    x_{j+1/2} \coloneq (j+1/2)\Delta x, \; 
    I_j \coloneq \open{x_{j-1/2}}{x_{j+1/2}}.
\]

Discretize the initial datum and “the real line”:
\[
    \forall j \in \Z, \;  
    u_j^o \coloneq \frac{1}{\Delta x} \int_{I_j} u_o(x) \d{x}, \quad
    \text{for a.e. } x \in \R, \; 
    \xi_\Delta(x) \coloneq \sum_{j \in \Z} x_j \1_{I_j}(x).
\]

Fix $n \in \N$. Let us explain how, given $(u_j^n)_{j \in \Z}$, we determine  
$(u_j^{n+1})_{j \in \Z}$. Set $u^n \coloneq \sum_{j \in \Z} u_j^n \1_{I_j}$.
Let us call $\cU^n \in \Czero([t^n, t^{n+1}], \Lloc{1}(\R))$ the unique entropy solution to following Cauchy problem, set on $\open{t^n}{t^{n+1}} \times \R$:
\begin{empheq}[left=\empheqlbrace]{align}
    \label{eq:SchemeStepOne}
    \p_t \cU(t, x) + \p_x \bigl(H(\xi_\Delta(x), \cU(t, x)) \bigr) & = 0 \nonumber \\
    \cU(t^n, x) & = u^n(x).
\end{empheq}

We recall in Section \ref{ssec:DiscontinuousFlux} the meaning of solutions to 
conservation laws with discontinuous flux we use and also establish that 
\eqref{eq:SchemeStepOne} admits a unique solution, which justifies the definition of $\cU^n$. Notice that since the space dependency is localized in a compact subset of $\R$, outside this compact, \eqref{eq:SchemeStepOne} reduces to a standard homogeneous conservation law. It then makes sense to $u_j^{n+1}$ as 
\begin{equation}
    \label{eq:MF1}
    \forall j \in \Z, \quad u_j^{n+1} \coloneq 
    \frac{1}{\Delta x} \int_{I_j} \cU^n(t^{n+1}-, x) \d{x},
\end{equation}

where $\cU^n(t^{n+1}-) \in \Lloc{1}(\R)$ denotes the left-hand limit of $\cU^n$ at $t=t^{n+1}$.

For the numerical analysis, we define the approximate solution $u_\Delta$ the 
following way:
\[
    u_\Delta \coloneq \sum_{n=0}^{+\infty} \cU^n \1_{[t^n ,t^{n+1}[}. 
\]

We can now state our main result.

\begin{theorem}
    \label{th:main}
    Assume that $H$ satisfies \eqref{eq:smoothness}--\eqref{eq:LocalizedSpace}--\eqref{eq:convexity} and fix 
    $u_o \in \L{\infty}(\R, \R)$. There exist two constants $\underline{u}$, 
    $\overline{u}$ depending only on $H$ and $u_o$ such that under the CFL condition 
    \begin{equation}
        \label{eq:CFL}
        2 \lambda L \leq 1, \quad 
        \lambda \coloneq \frac{\Delta t}{\Delta x}, \;
        L\coloneq \sup_{\underset{\underline{u} \leq p \leq \overline{u}}{x \in \R}} 
        \abs{\p_u H(x, p)},
    \end{equation}

    the following points hold.

    \begin{enumerate}[label=\bf(M\arabic*)]
        \item \label{item:Stability} 
        The scheme is stable:
        \begin{equation}
            \label{eq:StabilityScheme}
            \text{for a.e.} \; (t, x) \in \open{0}{+\infty} \times \R, 
            \quad \underline{u} \leq u_\Delta(t, x) \leq \overline{u}.
        \end{equation}

        \item \label{item:Convergence} As $(\Delta x, \Delta t) \to (0, 0)$, 
        the sequence $(u_\Delta)_\Delta$ converges in $\Lloc{p}(\open{0}{+\infty} \times \R, \R)$ for all $p \in [1, +\infty \mathclose[$ and a.e. on 
        $\open{0}{+\infty} \times \R$ to the unique entropy solution of \eqref{eq:CL}.
    \end{enumerate}
\end{theorem}

The main contribution of the paper is \ref{item:Stability}, which proof is detailed in 
Section \ref{ssec:SteadyStates}. Indeed, when building a finite volume scheme for 
a space dependent conservation law, the main difficulty to overcome in the convergence 
analysis is the obtaining of \textit{a priori} $\L{\infty}$ bounds for the 
approximating sequence. 
The input of the discontinuous flux theory is that we can build “by hand” discrete 
steady states of the scheme, making it, in a way, well-balanced.

The proof of \ref{item:Convergence} is then somewhat standard in the sense that if follows the classical steps of finite volume schemes analysis. All the details are presented in 
Section \ref{ssec:Convergence}. The technical detail in the convergence proof is the use of the compensated compactness method \cite{LuBook,SerreBook,Tartar1979,Tartar1982}, since global $\BV$ bounds are not expected in the context of discontinuous flux. The genuine nonlinearity of $H$ under Assumption \eqref{eq:convexity} is exploited.

\begin{remark}
    A byproduct of the analysis in that under the assumptions of Theorem \ref{th:main}, the sequence
    \[
        \rho_\Delta \coloneq \sum_{n=0}^{+\infty} \sum_{j \in \Z} 
        u_j^n \1_{[t^n, t^{n+1} \mathclose[ \times I_j}
    \]

    also converges towards the entropy solution of \eqref{eq:CL}, see Step 1 in Section \ref{ssec:Convergence}. This sequence is meaningful since it is the one actually implemented in a computer.
\end{remark}

\section{Proofs}

To fix the ideas, the proofs are written in the convex case of 
Assumption \eqref{eq:convexity}.

\subsection{Discontinuous Flux}
\label{ssec:DiscontinuousFlux}

We recall and adapt for our use results regarding conservation laws with discontinuous flux, borrowing elements from \cite{AKR2011}. For the purpose of this section, let us fix $f_l, f_r \in \Ck{1}(\R, \R)$ two convex functions satisfying \eqref{eq:convexity}. Consider:
\begin{equation}
    \label{eq:CLDiscontinuous}
    \p_t u + \p_x \left(F(x, u) \right) = 0, \quad
    F(x, u) \coloneq \left\{ 
        \begin{aligned}
            f_l(u) & \quad \text{if} \; x \leq 0 \\
            f_r(u) & \quad \text{if} \; x > 0.
        \end{aligned}
    \right.
\end{equation}

\begin{remark}
    This structure will be applied in the scheme at each interface with $f_l(u) = H(x_j, u)$ and $f_r(u) = H(x_{j+1}, u)$.    
\end{remark}

Since the works of \cite{AG2003, AJG2004,AGM2005,AKR2011, Diehl1995,GR1992,KRT2003,KT2004,Towers2000}, it is now known that additionally to the conservation of mass (Rankine-Hugoniot condition)
\[
    \text{for a.e.} \; t > 0, \quad f_l(u(t, 0-)) = f_r(u(t, 0+)),
\]

an entropy criterion must be imposed at the interface $\{x=0\}$ 
to select one solution. This choice is often guided by physical consideration. 
Among the most common criteria, we can cite: the minimal jump condition 
\cite{GR1992, KMRWellBalanced,KMRTriangular}, the vanishing viscosity criterion 
\cite{BGS2019}, or the flux maximization \cite{AGM2005}.

We call $\Phi_l$, resp. $\Phi_r$, the Kruzhkov entropy flux \cite{Kruzhkov1970} associated with $f_{l}$, resp. with $f_r$, meaning that
\[
    \forall a, b \in \R, \quad 
    \Phi_l(a, b) \coloneq \sgn(a-b) (f_l(a) - f_l(b)).
\]

In this section, define $\Phi = \Phi(x, u)$ as the Kruzhkov entropy flux associated with $F$:
\[
    \forall x \in \R, \; \forall a, b \in \R, \quad 
    \Phi(x, a, b) \coloneq \sgn(a-b) (F(x, a) - F(x, b)) 
    = \left\{ 
        \begin{aligned}
            \Phi_l(a, b) & \quad \text{if} \; x \leq 0 \\
            \Phi_r(a, b) & \quad \text{if} \; x > 0.
        \end{aligned}
    \right.
\]

For the study of \eqref{eq:CLDiscontinuous}, we follow \cite{AKR2011}, where the traces 
of the solution at the interface $\{x=0\}$ are explicitly treated. It will be useful to 
have a name for the critical points of $f_{l,r}$, say $\alpha_{l, r}$.
Notice that for all $y \in \open{\min f_r}{+\infty}$, the equation $f_r(u) = y$ admits exactly two solutions, $S_r^-(y) < \alpha_r < S_r^+(y)$. When $y=\min f_r$, then $S_r^-(y) = S_r^+(y) = \alpha_r$. This motivates the following definition.

\begin{definition}
    \label{def:Germ}
    Define $\cG \subset \R^2$ as the union of the following subsets:
    \begin{equation}
        \label{eq:Germ}
        \begin{aligned}
            \cG_1 & \coloneq \left\{ (k_l, k_r) \in \R^2 \; : \; k_l \geq \alpha_l, \; 
            k_r = S_r^+(f_l(k_l)) \right\} \\
            \cG_2 & \coloneq \left\{ (k_l, k_r) \in \R^2 \; : \; k_l \leq \alpha_l, \; 
            k_r = S_r^-(f_l(k_l)) \right\} \\
            \cG_3 & \coloneq \left\{ (k_l, k_r) \in \R^2 \; : \; k_l > \alpha_l, \; 
            k_r = S_r^-(f_l(k_l)) \right\}.
        \end{aligned}
    \end{equation} 
\end{definition}

We will show that $\cG$ is a germ. The term was introduced in \cite{AKR2011} and is used to describe an abstract subset having certain properties. The germ contains all the possible traces along $\{x=0\}$ of the solutions to \eqref{eq:CLDiscontinuous}. Notice that by construction, any couple in the germ satisfies the Rankine-Hugoniot condition. Conversely, some couples verifying the Rankine-Hugoniot condition have been excluded from the germ, more precisely the ones 
belonging to 
\[
    \left\{ (k_l, k_r) \in \R^2 \; : \; k_l < \alpha_l, \; k_r = S_r^+(f_l(k_l)) \right\}.
\]

The reason lies in the following proposition, see in particular \eqref{eq:MaximalGerm}.

\begin{proposition}
    \label{pp:GermL1DMax}
    $\cG$ defined in Definition \ref{def:Germ} is a maximal $\L{1}$-dissipative germ, 
    meaning that the following properties are satisfied.

    (i) For all $(u_l, u_r) \in \cG$, $f_l(u_l) = f_r(u_r)$.

    (ii) Dissipative inequality: for all $(u_l, u_r) \in \cG$ and $(k_l, k_r) \in \cG$, it holds:
    \begin{equation}
        \label{eq:DissipativeGerm}
        \Phi_l(u_l, k_l) - \Phi_r(u_r, k_r) \geq 0.
    \end{equation}
    
    (iii) Maximality condition: let $(u_l, u_r) \in \R^2$ such that $f_l(u_l) = f_r(u_r)$. Then 
    \begin{equation}
        \label{eq:MaximalGerm}
        \bigl(  \forall (k_l, k_r) \in \cG, \; 
        \Phi_l(u_l, k_l) - \Phi_r(u_r, k_r) \geq 0 \bigr) \implies (u_l, u_r) \in \cG.
    \end{equation}
\end{proposition}

\begin{proof}
    \textit{(i)} Immediate from the construction.
    
    \textit{(ii)} Follows from a case by case study. Let us deal with one case. Assume for instance that $(u_l, u_r) \in \cG_1$ and $(k_l, k_r) \in \cG_3$. Recall that $u_r \geq \alpha_r \geq k_r$ and $u_l, k_l \geq \alpha_l$. If $u_l \geq k_l$, then $\Phi_l(u_l, k_l) - \Phi_r(u_r, k_r)= 0$. Otherwise, using the Rankine-Hugoniot condition,
    \[
        \Phi_l(u_l, k_l) - \Phi_r(u_r, k_r) = 2 (f_l(k_l) - f_l(u_l)) > 0, 
    \] 

    where we used the monotonicity of $f_l$. The proof for the other cases follows the same idea.
    
    \textit{(iii)} We prove it by way of contradiction: let $(u_l, u_r) \in \R^2$ such that $f_l(u_l) = f_r(u_r)$ and assume that 
    \[
        \forall (k_l, k_r) \in \cG, \; \Phi_l(u_l, k_l) - \Phi_r(u_r, k_r) \geq 0 
        \quad \text{and} \quad (u_l, u_r) \notin \cG.
    \]

    From the Rankine-Hugoniot condition, we deduce $u_r = S_r^+(f_l(u_l))$, with 
    $u_l < \alpha_l$. Take $k_r \coloneq S_r^-(f_l(k_l))$ with $k_l \in \open{u_l}{\alpha_l}$. 
    Clearly, $(k_l, k_r) \in \cG_2 \subset \cG$, therefore by assumption, $\Phi_l(u_l, k_l) - \Phi_r(u_r, k_r) \geq 0$. But the choice of $k_l$ raises the contradiction:
    \[
        \Phi_l(u_l, k_l) - \Phi_r(u_r, k_r)
        = (f_l(k_l)-f_l(u_l)) - (f_r(u_r) - f_r(k_r)) 
        = 2 (f_l(k_l)-f_l(u_l)) < 0.
    \]
\end{proof}

The theory of \cite{AKR2011} proposes an abstract framework for the study 
of \eqref{eq:CLDiscontinuous}. On the other hand, the authors of 
\cite{AGM2005, AJG2004} give a formula for the flux at the interface for Riemann 
problems of \eqref{eq:CLDiscontinuous}. In the following proposition, we link the two 
points of view. Let us adopt the notations:
\begin{equation}
    \label{eq:MinMax}
    \forall a, b \in \R, \quad a \wedge b \coloneq \min\{a, b\} 
    \quad \text{and} \quad a \vee b \coloneq \max\{a, b\}.
\end{equation}

\begin{proposition}
    \label{pp:FluxAJG}
    For all $(u_l, u_r) \in \R^2$, define the interface flux:
    \begin{equation}
        \label{eq:FluxAJG}
        \fint(u_l, u_r) \coloneq 
        \max\{f_l(u_l \vee \alpha_l), f_r(\alpha_r \wedge u_r)\},
    \end{equation}

    and the remainder term:
    \begin{equation}
        \label{eq:RemainderTerm}
        \cR(u_l, u_r) \coloneq 
        \abs{\fint(u_l, u_r) - f_l(u_l)} + \abs{\fint(u_l, u_r) - f_r(u_r)}.
    \end{equation}
    
    Then the following points hold.

    (i) For all $(u_l,u_r) \in \R^2$, $(u_l,u_r) \in \cG \iff \cR(u_l,u_r) = 0$.

    (ii) For all $(u_l, u_r) \in \R^2$ and for all $(k_l, k_r) \in \cG$, 
    $\Phi_r(u_r, k_r) - \Phi_l(u_l, k_l) \leq \cR(u_l, u_r)$. 
\end{proposition}

\begin{proof}
    The fact that we have an “explicit” description of $\cG$ facilitates the proof.

    \textit{(i)} If $(u_l, u_r) \in \cG_1 \cup \cG_2$, then $u_l - \alpha_l$ and 
    $u_r -\alpha_r$ have the same sign, therefore 
    \[
        \fint(u_l, u_r) = 
        \left\{ 
            \begin{aligned}
                f_l(u_l)  & \quad \text{if} \; u_l \geq \alpha_l \\
                f_r(u_r) & \quad \text{if} \; u_l < \alpha_l
            \end{aligned}
        \right.
        = f_l(u_l) = f_r(u_r).
    \]

    On the other hand, if $(u_l, u_r) \in \cG_3$, then 
    \[
        \fint(u_l, u_r) = \max\{f_l(u_l), f_r(u_r)\} = f_l(u_l) = f_r(u_r).
    \]

    This ensures that $\cR(u_l, u_r) = 0$.

    Conversely, assume that $\cR(u_l, u_r) = 0$. In particular, the 
    Rankine-Hugoniot holds. We deduce that 
    \[
        \cR(u_l, u_r) = 0 
        \implies (u_l, u_r) \in \cG \;\; \text{or} \;\; 
        u_r = S_r^+(f_l(u_l)), \; u_l < \alpha_l.
    \]

    But in the latter case, $\fint(u_l, u_r) = \max\{\min f_l, \min f_r \} 
    \neq f_l(u_l)$. Therefore, $(u_l,u_r) \in\cG$.
    \smallskip 

    \textit{(ii)} Set $q\coloneq \Phi_r(u_r, k_r) - \Phi_l(u_l, k_l)$. Keep in mind that $f_{l,r}$ is decreasing on $\mathopen]-\infty, \alpha_{l,r}]$ and increasing on $[\alpha_{l, r}, +\infty \mathclose[$. We proceed with a 
    case by case study.
    \smallskip 

    \textbf{Case 1:} $u_l \leq k_l$ and $u_r \leq k_r$. Then 
    $q = f_l(u_l) - f_r(u_r) \leq \cR(u_l, u_r)$.

    \textbf{Case 2:} $u_l > k_l$ and $u_r > k_r$. Similar to Case 1.

    \textbf{Case 3:} $u_l > k_l$ and $u_r \leq k_r$. Then 
    $q = (f_r(k_r) - f_r(u_r)) + (f_l(k_l) - f_l(u_l))$. We proceed by a case by case study, depending on which part of the germ $(k_l, k_r)$ belongs to. Notice that here, we have $\cR(u_l, u_r) = \abs{f_l(u_l) - f_r(u_r)}$.
    
    3.(a) $(k_l, k_r) \in \cG_1$. Since $u_l > k_l \geq \alpha_l$, we write, using the monotonicity and $f_l(k_l) = f_r(k_r)$:
    \[
        q \leq f_r(k_r) - f_r(u_r)
        = (f_l(u_l) - f_r(u_r)) + (f_l(k_l) - f_l(u_l)) 
        \leq f_l(u_l) - f_r(u_r) \leq \cR(u_l, u_r).
    \]

    3.(b) $(k_l, k_r) \in \cG_2$. This time, $u_r \leq k_r \leq \alpha_r$, therefore:
    \[
        q \leq f_l(k_l) - f_l(u_l) = (f_r(u_r) - f_l(u_l)) + (f_r(k_r) - f_r(u_r)) 
        \leq f_r(u_r) - f_l(u_l) \leq \cR(u_l, u_r).
    \] 

    3.(c) $(k_l, k_r) \in \cG_3$. Clearly, $q \leq 0$.
    
    \textbf{Case 4:} $u_l \leq k_l$ and $u_r > k_r$. Similar to Case 3, so the details are omitted.
\end{proof}

\begin{remark}
    \label{rk:ConstantRemainderTerm}
    Notice that for all $k \in \R$,
    \[
        \cR(k, k) = 
        \left\{ 
            \begin{aligned}
                \abs{\max\{\min f_l, \min f_r \}-f_l(k)} 
                + \abs{\max\{\min f_l, \min f_r \}-f_r(k)} 
                & \quad \text{if} \; \alpha_r \leq k < \alpha_l \\
                \abs{f_l(k) - f_r(k)} & \quad \text{otherwise.}
            \end{aligned}
        \right.
    \] 
\end{remark}

We are now in position to properly give the definition of solutions for 
\eqref{eq:CLDiscontinuous}. The first definition is a Kruzhkov type entropy inequality where the solution is tested not against constant data, but against piecewise constant data. The second definition relies on the treatment of the traces of the solution along $\{x=0\}$. More precisely, for $u \in \L{\infty}(\open{0}{+\infty} \times \R)$, the strong traces of $u$ along $\{x=0\}$, if they exist, are the functions $\gamma_{L}u, \gamma_{R}u \in \L{\infty}(\open{0}{+\infty}, \R)$ such that
\[
    \underset{r \to 0+}{\mathrm{ess \; lim}} \; \frac{1}{r} \int_{-r}^0 \abs{u(t, x) - \gamma_{L}u(t)} \d{x} = 0, \quad \underset{r \to 0+}{\mathrm{ess \; lim}} \; \frac{1}{r} \int_0^r \abs{u(t, x) - \gamma_{R}u(t)} \d{x} = 0.
\]

\begin{proposition}
    \label{pp:EquivalentDefinitions}
    Define $\cG$ as \eqref{eq:Germ} and $\cR$ as \eqref{eq:RemainderTerm}. 
    Let $u_o \in \L{\infty}(\R, \R)$ and $u \in \L{\infty}(\open{0}{+\infty} \times \R, \R)$. Then the statements $\mathbf{(A)}$ and $\mathbf{(B)}$ below are equivalent.

    $\mathbf{(A)}$ For all $\varphi \in \Cc{\infty}(\R^+ \times \R, \R^+)$ and for 
    all $(k_l, k_r) \in \R^2$, with $\kappa \coloneq k_l \1_{\open{-\infty}{0}} + k_r \1_{\open{0}{+\infty}}$, it holds:
    \begin{empheq}{align}
    \label{eq:DiscontinuousEI1}
    \int_0^{+\infty} \int_\R 
    \abs{u - \kappa(x)} \p_t \varphi + \Phi(x, u, \kappa(x)) \p_x \varphi 
    \; \d{x} \d{t} 
	+ \int_\R \abs{u_o(x) - \kappa(x)} \varphi(0, x) \d{x} \nonumber \\ 
	+ \int_0^{+\infty} \cR(k_l, k_r) \varphi(t, 0) \d{t} \geq 0.
    \end{empheq}

    $\mathbf{(B)}$ The two following points are verified.

    (i) For all $\varphi \in \Cc{\infty}(\R^+ \times \R \backslash \{0\}, \R^+)$ 
    and for all $(k_l, k_r) \in \R^2$, with $\kappa \coloneq k_l \1_{\open{-\infty}{0}} + k_r \1_{\open{0}{+\infty}}$, it holds:
    \begin{empheq}{align}
    \label{eq:DiscontinuousEI2}
    \int_0^{+\infty} \int_\R 
    \abs{u - \kappa(x)} \p_t \varphi + \Phi(x, u, \kappa(x))  \p_x \varphi \; \d{x} \d{t} 
    + \int_\R \abs{u_o(x) - \kappa(x)} \varphi(0, x) \d{x} \geq 0.
    \end{empheq}

    (ii) For a.e. $t \in \open{0}{+\infty}$, $(\gamma_L u (t), \gamma_R u (t)) \in \cG$.
    \smallskip 

    When $\mathbf{(A)}$ or $\mathbf{(B)}$ holds, we say that $u$ is an entropy solution to \eqref{eq:CLDiscontinuous} with initial datum $u_o$.
\end{proposition}

\begin{proof}
    See the proof of \cite[Theorem 3.18]{AKR2011}. 
    It is worth mentioning that either \eqref{eq:DiscontinuousEI1} or \eqref{eq:DiscontinuousEI2} implies that $u$ admit strong traces along $\{x=0\}$. Indeed, since $f_l$ and $f_r$ are genuinely nonlinear in the sense that
    \[
        \forall s \in \R, \quad \text{meas} \bigl(\left\{p \in \R \; : \; 
        f_{l,r}'(p) = s \right\} \bigr) = 0,
    \]

    existence of strong traces is obtained from the works of \cite{Vasseur2001, NPS2018}.
\end{proof}

We now state the well-posedness result regarding \eqref{eq:CLDiscontinuous}.

\begin{theorem}
    \label{th:WellPosednessDiscontinuous}
    Let $f_l, f_r \in \Ck{1}(\R, \R)$ verify \eqref{eq:convexity}.
    \begin{enumerate}[label=\bf(D\arabic*)]
        \item \label{item:WellPosednessDiscontinuous} 
        Fix $u_o \in \L{\infty}(\R, \R)$. Then there exist a unique entropy solution to 
        \eqref{eq:CLDiscontinuous} with initial datum $u_o$.

        \item \label{item:StabilityDiscontinuous} 
        Let $u_o, v_o \in \L{\infty}(\R, \R)$. We denote by $u$, resp. $v$, the entropy 
        solution to \eqref{eq:CLDiscontinuous} with initial datum $u_o$, resp. $v_o$. 
        Then, there exists $L > 0$ such that for all $R > 0$ and for all $t > 0$,
        \begin{equation}
            \label{eq:StabilityDiscontinuous}
            \begin{aligned}
                \int_{\abs{x} \leq R} \abs{u(t, x) - v(t, x)} \d{x}
                & \leq \int_{\abs{x} \leq R+Lt} \abs{u_o(x) - v_o(x)} \d{x} \\[5pt]
                \int_{\abs{x} \leq R} (u(t, x) - v(t, x))^+ \d{x}
                & \leq \int_{\abs{x} \leq R+Lt} (u_o(x) - v_o(x))^+ \d{x}.
            \end{aligned}
        \end{equation}
    \end{enumerate}
\end{theorem}

\begin{proof}
    The proofs of the stability bounds can be found in 
    \cite[Theorems 3.11-3.19]{AKR2011}. Regarding the existence, one can build a simple 
    finite volume scheme and prove its convergence, see for instance \cite{AJG2004}. 
    The flux to put at the interface $\{x = 0\}$ in the scheme is precisely $\fint$. 
\end{proof}

We conclude this section explaining how to solve the Riemann problem for 
\eqref{eq:CLDiscontinuous}. Indeed, the structure of the solution to Riemann problems 
will be used in the convergence analysis, see in particular 
Section \ref{sssec:CompComp}.

Given $(u_l, u_r)\in\R^2$, we solve \eqref{eq:CLDiscontinuous} with the initial 
condition 
\[
    u_o(x) = 
    \left\{ 
        \begin{array}{ccl}
        u_l & \text{if} & x <0 \\
        u_r & \text{if} & x>0.
        \end{array}
    \right.
\]

First, compute the interface flux given by \eqref{eq:FluxAJG}. We have:
\begin{equation}
    \label{eq:InterfaceFlux}
    \fint(u_l, u_r) = 
    \left\{ 
        \begin{array}{ccl}
        \max\{ \min f_l, f_r(u_r) \} & \text{if} 
        & u_l \leq \alpha_l \; \text{and} \; u_r\leq \alpha_r \quad \text{(I)} \\
        \max\{\min f_l, \min f_r\} & \text{if} 
        & u_l \leq \alpha_l \; \text{and} \; u_r > \alpha_r \quad \text{(II)} \\
        \max \{f_l(u_l), f_r(u_r) \} & \text{if} 
        & u_l > \alpha_l \; \text{and} \; u_r \leq \alpha_r \quad \text{(III)} \\
        \max \{f_l(u_l), \min f_r \} & \text{if} 
        & u_l > \alpha_l \; \text{and} \; u_r > \alpha_r \quad \text{(IV)}.
        \end{array}
    \right.
\end{equation}

The value of $\fint(u_l, u_r)$ imposes the value of one of the traces at the interface. 
For instance, if $\fint(u_l, u_r) = f_r(u_r)$, then the right-hand trace of the solution along $\{x=0\}$ is equal to $\gamma_r \coloneq u_r$. To compute the left-hand trace, one solves $f_l(\gamma) = f_r(\gamma_r)$, $\gamma \in \R$. This equation admits two solutions, $\gamma_l^- \leq \alpha_l \leq \gamma_l^+$, but only one of them is such that the classical Riemann problem with flux $f_l$ and states $(\gamma_l, u_r)$ only displays waves of negative speeds. This determines the left-hand trace. Finally, the solution to the Riemann problem is obtained by patching together:
\begin{itemize}
    \item the solution to the classical Riemann problem with flux $f_l$ and states $(u_l, \gamma_l)$ on $\open{0}{+\infty} \times \R^-$ ;
    \item the stationary (non entropic) shock connecting $(\gamma_l, \gamma_r)$ at $x = 0$ ;
    \item the solution to the classical Riemann problem with flux $f_r$ and states $(\gamma_r, u_r)$ on $\open{0}{+\infty} \times \R^+$ .
\end{itemize}

As an illustration, consider the data
\[
    f_l(u) \coloneq \frac{u^2}{2}, \quad 
    f_r(u) \coloneq u^2, \quad
    u_o^{(1)} \equiv -1, \quad
    u_o^{(2)}(x) \coloneq
    \left\{ 
        \begin{array}{ccl}
        -1 & \text{if} & x <0 \\
        1 & \text{if} & x>0.
        \end{array}
    \right.
\]

The values of $u_o^{(1)}$ lead to Case (I). The interface flux is $f_r(-1)$, therefore the right-hand trace of the solution is $\gamma_r = -1$. Then,
\[
    f_l(\gamma) = f_r(\gamma_r) \iff \gamma \in \{-\sqrt{2}, \sqrt{2}\}.
\]

The left-hand trace of the solution is $\gamma_l \coloneq -\sqrt{2}$ because the solution to the Riemann problem with flux $f_l$ and states $(-1, \sqrt{2})$ is a rarefaction wave presenting a positive speed. Consequently, the solution is 
\[
    u^{(1)}(t, x) = 
    \left\{ 
        \begin{array}{ccl}
        -1 & \text{if} & x < -\frac{t}{2(\sqrt{2}-1)} t \\
        -\sqrt{2} & \text{if} & -\frac{t}{2(\sqrt{2}-1)} t < x < 0 \\
        -1 & \text{if} & x > 0.
        \end{array}
    \right.
\]

On the other hand, the values of $u_o^{(2)}$ lead to Case (II). The interface flux is $0 = f_l(0) = f_r(0)$. We deduce that both traces of the solution at the interface are equal to 0. Therefore, the solution is obtained by patching together two rarefaction waves:
\[
    u^{(2)}(t, x) = 
    \left\{ 
        \begin{array}{ccl}
        -1 & \text{if} & x < -t \\
        \frac{x}{t} & \text{if} & -t < x < 0 \\
        \frac{x}{2t} & \text{if} & 0 < x < 2t \\
        1 & \text{if} & x > 2t.
        \end{array}
    \right.
\]

Below are listed the possible structures of the solution to a Riemann problem with discontinuous flux.

\paragraph{If $(u_l, u_r) \in \cG$,} 
then the solution is a stationary non-classical shock (SNS).

\paragraph{Case (I)} The solution is the gluing of 
shock/rarefaction - SNS - constant $u_r$ or rarefaction - SNS - shock.

\paragraph{Case (II)} The solution is the gluing of 
shock/rarefaction - SNS - rarefaction 
or rarefaction - SNS - shock/rarefaction. 

\paragraph{Case (III)} The solution is the gluing of 
shock - SNS - constant $u_r$, or constant $u_l$ - SNS - shock. 

\paragraph{Case (IV)} The solution is the gluing of 
shock-SNS-rarefaction, or constant $u_l$ - SNS - shock/rarefaction.

In any case, the solution presents at most one stationary non-classical shock and at 
most one classical shock.
\bigskip

The content of this section can be adapted in a straightforward way to the case 
where $F$ in \eqref{eq:CLDiscontinuous} presents a finite number of space 
discontinuities.

\subsection{Steady states and stability}
\label{ssec:SteadyStates}

This section is devoted to the proof of the $\L{\infty}$ stability of the scheme. 

\begin{lemma}
    \label{lmm:CompactHetero}
    Assume that $H$ satisfies \eqref{eq:smoothness}--\eqref{eq:LocalizedSpace}--\eqref{eq:convexity}. Then, 
    there exists a unique function $\alpha \in \Ck{2}(\R, \R)$ such that for 
    all $x \in \R$, $\p_u H(x, \alpha(x)) = 0$. Moreover, 
    \[
        \forall x \in \R, \quad \abs{x} \geq X \implies \alpha'(x) = 0.
    \]
\end{lemma}

\begin{proof}
    Straightforward application of the Implicit Function Theorem.    
\end{proof}

Introduce $J \in \N^*$ such that $X \in I_J$. Recall that the space dependency of $H$ 
is localized in $[-X, X]$.

Set for all $j \in \Z$, $\alpha_j \coloneq \alpha(x_j)$ and notice that
\[
    (j \geq J+1 \implies \alpha_j = \alpha(X) )\quad \text{and} \quad 
    (j \leq -J-1 \implies \alpha_j = \alpha(-X)).
\]

It will be convenient to adopt the notation 
\[
    \forall j \in \Z, \quad h_j(u) \coloneq H(x_j, u).
\]

Let us precise that for sufficiently small time interval $\Delta t$, the waves 
emanating at each interface $x = x_{j+1/2}$ do not interact. Consequently, 
\eqref{eq:MF1} rewrites as
\begin{equation}
    \label{eq:MF2}
    \forall j \in \Z, \quad u_j^{n+1} 
    = u_j^n - \lambda (\fint^{j+1/2} (u_{j}^n, u_{j+1}^n) 
    - \fint^{j-1/2} (u_{j-1}^n, u_{j}^n)),
\end{equation}

where following Section \ref{ssec:DiscontinuousFlux},
\[
    \fint^{j+1/2}(a, b) \coloneq 
    \max \biggl\{\god_j(a, \alpha_j), \; \god_{j+1}(\alpha_{j+1}, b) \biggr\}
\]

is the flux across the interface $x = x_{j+1/2}$, with $\god_j$ 
denoting the Godunov flux associated with $h_j$.

It is worth mentioning that for all $j \in \Z$, $\fint^{j+1/2}$ is a monotone, locally 
Lipschitz, numerical flux, see \cite[Section 4]{AJG2004}. Also, remark that in 
general, $\fint^{j+1/2}$ is not consistent, meaning that for all $k\in\R$, 
$\fint^{j+1/2}(k, k)$ is not equal to $h_j(k)$ or to $h_{j+1}(k)$. However, using the 
regularity of $H$, we can prove that a consistency holds, at the limit.

\begin{lemma}
    \label{lmm:ConsistencyInterfaceFlux}
    Assume that $H$ satisfies \eqref{eq:smoothness}--\eqref{eq:LocalizedSpace}--\eqref{eq:convexity}. Then for any 
    $k \in \R$, there exists a positive constant $\mu_k$ such that for all $j \in \Z$,
    \[
        \abs{\fint^{j+1/2}(k, k) - h_{j+1}(k)} \leq \mu_k \; \Delta x.
    \]
\end{lemma}

\begin{proof}
    Let $j \in \Z$. Recall that for all $k \in \R$,
    \[
        \fint^{j+1/2}(k, k) =
        \left\{ 
            \begin{array}{cclc}
            \max\{ h_j(\alpha_j), h_{j+1}(k) \} & \text{if} 
            & k \leq \alpha_j, \alpha_{j+1} \quad & \text{(I)} \\
            \max\{ h_j(\alpha_j), h_{j+1}(\alpha_{j+1}) \} & \text{if} 
            & \alpha_{j+1} \leq k \leq \alpha_j \quad & \text{(II)} \\
            \max\{ h_j(k), h_{j+1}(k) \} & \text{if} 
            & \alpha_{j} \leq k \leq \alpha_{j+1} \quad & \text{(III)} \\
            \max\{ h_j(k), h_{j+1}(\alpha_{j+1}) \} & \text{if} 
            & k > \alpha_j, \alpha_{j+1} \quad & \text{(IV)}
            \end{array}
        \right.
    \]

    The proof reduces to a case by case study. Set 
    $\delta \coloneq \abs{\fint^{j+1/2}(k, k) - h_{j+1}(k)}$.

    \textbf{Case (I)} If $\fint^{j+1/2}(k, k) = h_{j+1}(k)$, then $\delta = 0$. 
    Otherwise, $\fint^{j+1/2}(k, k) = h_{j}(\alpha_j)$, meaning that 
    $h_j(\alpha_j) \geq h_{j+1}(k)$. Therefore,
    \[
        \delta = h_{j}(\alpha_j) - h_{j+1}(k) 
        \leq h_{j}(k) - h_{j+1}(k)
        \leq \sup_{x \in \R} \abs{\p_x H(x, k)} \; \Delta x.
    \]

    \textbf{Case (II)} In that case, the estimate follows from the fact that 
    \[
        \abs{h_j(\alpha_j) - h_{j+1}(k)} \leq 
        \bigl( \norm{\alpha'}_{\L{\infty}(\R)} 
        \sup_{\underset{\abs{p} \leq \norm{\alpha}_{\L{\infty}(\R)}}{x \in \R}} \abs{\p_u H(x, p)} + \sup_{x \in \R} \abs{\p_x H(x, k)} \bigr) \Delta x.
    \]

    \textbf{Case (III)} In that case, we can bound $\delta$ by 
    $\sup_{x \in \R} \abs{\p_x H(x, k)} \; \Delta x$.

    \textbf{Case (IV)} If $\fint^{j+1/2}(k, k) = h_{j}(k)$, then we bound $\delta$ by 
    $\sup_{x \in \R} \abs{\p_x H(x, k)} \; \Delta x$. 
    Otherwise, $\fint^{j+1/2}(k, k) = h_{j+1}(\alpha_{j+1})$, meaning that 
    $h_{j+1}(\alpha_{j+1}) \geq h_{j}(k)$. Therefore,
    \[
        \begin{aligned}
            \delta 
            & \leq \abs{h_{j+1}(\alpha_{j+1}) - h_j(k)} + \abs{h_j(k) - h_{j+1}(k)} \\
            & = (h_{j+1}(\alpha_{j+1}) - h_j(k)) + \abs{h_j(k) - h_{j+1}(k)} \\
            & \leq 2 \abs{h_j(k) - h_{j+1}(k)} \\
            & \leq 2 \sup_{x \in \R} \abs{\p_x H(x, k)} \; \Delta x.
        \end{aligned}
    \]
\end{proof}

To prove the $\L{\infty}$ bounds \ref{item:Stability}, we explicitly construct steady states of the scheme \eqref{eq:MF2}. 

\begin{definition}
    \label{def:SteadyState}
    We say that a sequence $(v_j)_j$ is a steady state of 
    \eqref{eq:MF2} if 
    \[
        \forall j \in \Z, \quad 
        \fint^{j-1/2} (v_{j-1}, v_{j}) = \fint^{j+1/2} (v_{j}, v_{j+1}).
    \]
\end{definition}

Following Section \ref{ssec:DiscontinuousFlux}, it will be useful to introduce for all 
$j \in \Z$ and $y \in \mathopen] \min h_j, +\infty \mathclose[$, the two solutions 
$S_{j}^-(y) < \alpha_j < S_j^+(y)$ of the equation $h_j(x) = y$, $x \in \R$. 
For all $j \in \Z$, the definition of the germ $\cG_{j+1/2}$ follows from 
Definition \ref{def:Germ}.

\begin{lemma}
    \label{lmm:SteadyState}
    Assume \eqref{eq:smoothness}--\eqref{eq:LocalizedSpace}--\eqref{eq:convexity} hold. Let $c \in \R$ such that $c \geq \sup_{\R} \alpha$. Define for all $j \in \Z$,
    \[
        v_{j+1} \coloneq 
        \left\{ 
        \begin{array}{ccl}
            c & \text{if} & j \leq -J-2 \\
            S_{j+1}^+(h_j(v_j)) & \text{if} & j \geq -J-1,
        \end{array}
        \right.  \quad \quad
        w_j \coloneq 
        \left\{ 
        \begin{array}{ccl}
            S_{j}^+ (h_{j+1}(w_{j+1})) & \text{if} & j \leq J \\
            c & \text{if} & j \geq J+1.
        \end{array}
        \right.
    \]

    Then $(v_j)_j$ and $(w_j)_j$ are steady states of \eqref{eq:MF2} 
    in the sense of Definition \ref{def:SteadyState}.
\end{lemma}

\begin{proof}
    Consider $(v_j)_j$. Since $c \geq \sup_{\R} \alpha$, for all $j \in \Z$, $(v_j, v_{j+1}) \in \cG_{j+1/2}$, see Definition \ref{def:Germ}. Indeed, we have $v_{j} \geq \alpha_j$, $v_{j+1} \geq \alpha_{j+1}$ and $h_{j}(v_j) = h_{j+1}(v_{j+1})$. Consequently, Proposition \ref{pp:FluxAJG} ensures that 
    \[
        h_j(v_j) = \fint^{j+1/2} (v_{j}, v_{j+1}) = h_{j+1}(v_{j+1}).
    \]

    The same reasoning holds for $(w_j)_j$.
\end{proof}

\begin{remark}
    Keep the notations of Lemma \ref{lmm:SteadyState}.
    
    (i) Notice that $(v_j)_j$ and $(w_j)_j$ are stationary. For instance, for all 
    integers $j \geq J+2$,
    \[
        h_{j-1}(v_{j-1}) = h_j(v_j) = h_{j-1}(v_j) \implies v_{j-1} = v_j.
    \] 

    (ii) For all $c \in \R$ such that $c \leq \inf_\R \alpha$, the two following 
    sequences are steady states as well: 
    \[
        v_{j+1} = 
        \left\{ 
        \begin{array}{ccl}
            c & \text{if} & j \leq -J-2 \\
            S_{j+1}^- (h_j(v_j)) & \text{if} & j \geq -J-1,
        \end{array}
        \right.  \quad \quad
        w_j = 
        \left\{ 
        \begin{array}{ccl}
            S_{j}^- (h_{j+1}(w_{j+1})) & \text{if} & j \leq J \\
            c & \text{if} & j \geq J+1.
        \end{array}
        \right.
    \]
\end{remark}

For the construction of the steady states, we will use geometric properties of both $H$ and its Legendre transform with respect to the second variable (see \cite[Appendix 4.2]{CannarsaSinestrariSBook}). We call it $L$ from now on.

\begin{lemma}
    \label{lmm:Legendre}
    Assume that $H$ satisfies \eqref{eq:smoothness}--\eqref{eq:LocalizedSpace}--\eqref{eq:convexity}. Then the following points hold. 

    (i) $L \in \Ck{2}(\R^2, \R)$ and satisfies \eqref{eq:LocalizedSpace}--\eqref{eq:convexity}.

    (ii) We have 
    \begin{equation}
        \label{eq:Legendre}
        \forall \lambda > 0, \; \forall x \in \R, \; \forall p \in \R, \quad
        \lambda \abs{p} - H(x, p) \leq \sup_{\underset{\abs{v} \leq \lambda}{y \in \R}} L(y, v).
    \end{equation}
\end{lemma}

\begin{proof}
    Recall that $L(x, v) \coloneq \sup_{p \in \R} (pv - H(x, p))$, see \cite[Appendix 4.2]{CannarsaSinestrariSBook}. The growth of $H$ ensures that $L$ is well-defined on $\R^2$. 

    \textit{(i)} The regularity of $L$ follows from an application of the Implicit Function Theorem and the fact that $L$ satisfies \eqref{eq:LocalizedSpace} is immediate from the definition. Finally, for all $x \in \R$, $v \mapsto \p_v L(x, v)$ is the reciprocal of $u \mapsto \p_u H(x, u)$, therefore, \eqref{eq:convexity} holds as well.

    \textit{(ii)} We first prove that \eqref{eq:Legendre} holds swapping the roles of $H$ and $L$. 
    
    For all $\lambda > 0$, $x \in \R$ and $v \in \R \setminus \{0\}$, specializing with $p = \frac{\lambda v}{\abs{v}}$ in the definition of $L$ yields:
    \[
        L(x, v) \geq \lambda \abs{v} - H(x, \frac{\lambda v}{\abs{v}}) 
        \geq \lambda \abs{v} - \sup_{\underset{\abs{p} \leq \lambda}{y \in \R}} H(y, p). 
    \]

    Inequality \eqref{eq:Legendre} follows from the same computations, using the fact that the Legendre transform is an involution on convex functions.
\end{proof}

We are now in position to construct steady states of the scheme.

\begin{lemma}
    \label{lmm:SteadyStateBis}
    Assume that $H$ satisfies \eqref{eq:smoothness}--\eqref{eq:LocalizedSpace}--\eqref{eq:convexity}.
    Let $(m, M) \in \R^2$ such that $m \leq \inf_{\R} \alpha$ and 
    $M \geq \sup_{\R} \alpha$. Then there exist two steady states of \eqref{eq:MF2}, 
    $(v_j)_j$ and $(w_j)_j$, such that:
    
    (i) $(v_j)_j$ is bounded from below and for all $j \in \Z$, $v_j \leq m$ ;

    (ii) $(w_j)_j$ is bounded from above and for all $j \in \Z$, $w_j \geq M$.
\end{lemma}

\begin{proof}
    We only give the details for the construction of $(w_j)_j$. Set 
    \[
        \cU_h \coloneq \max_{x \in \R} H(x, M) \quad \text{and} \quad 
        \overline{M} \coloneq \cU_h + \sup_{\underset{\abs{v} \leq 1}{x \in \R}} L(x, v).
    \]

    By monotonicity, for all $j \in \Z$, $S_j^+(\cU_h) \geq M$. Therefore, 
    taking advantage of Lemma \ref{lmm:Legendre} \textit{(ii)}:
    \[
        \forall j \in \Z, \quad 
        S_j^+(\cU_h) 
        \leq H(x_j, S_j^+(\cU_h)) + \sup_{\underset{\abs{v} \leq 1}{x \in \R}} L(x, v) 
        = \overline{M},
    \]

    ensuring that $\overline{M} \geq M \geq \sup_{\R} \alpha$. Now, define 
    \[
        \forall j \in \Z, \quad
        w_{j+1} \coloneq
        \left\{ 
        \begin{array}{ccl}
            \overline{M} & \text{if} & j \leq -J-2 \\
            S_{j+1}^+(h_j(w_j)) & \text{if} & j \geq -J-1.
        \end{array}
        \right.
    \]

    Lemma \ref{lmm:SteadyState} ensures that $(w_j)_j$ is a steady state of the scheme 
    since $\overline{M} \geq \sup_{\R} \alpha$.

    \paragraph{Claim 1:} for all $j \in \Z$, $w_j \geq M$. The key point is that $(h_j(w_j))_j$ is constant. Therefore, for all $j \in \Z$,
    \[
        h_j(w_j) 
        = h_{-J-1}(w_{-J-1}) = h_{-J-1}(\overline{M})
        \geq h_{-J-1}(S_{-J-1}^+(\cU_h)) = \cU_h \geq h_j(M),
    \]

    from which we deduce, by monotonicity, that $w_j \geq M$. The claim is proved.

    \paragraph{Claim 2:} $(w_j)_j$ is bounded from above. Once again, we use that $(h_j(w_j))_j$ is constant. For all $j \in \Z$, 
    \[
        \begin{aligned}
            h_j(w_j) = h_{-J-1}(\overline{M}) 
            & \implies w_j = S_j^+(h_{-J-1}(\overline{M})) \\
            & \implies w_j = S_j^+ \left(H(-X, \overline{M}) \right) \\
            & \implies w_j \leq H(-X, \overline{M}) + \sup_{\underset{\abs{v} \leq 1}{x \in \R}} L(x, v)\coloneq \overline{u} 
            \quad (\text{Lemma} \; \ref{lmm:Legendre} \; (ii)).
        \end{aligned}
    \]
    
    We see that $\overline{u}$ only depends on $H$ and $M$. Let us precise that one can choose 
    \[
        \underline{u} \coloneq -H(-X, \overline{m}) - \sup_{\underset{\abs{v} \leq 1}{x \in \R}} L(x, v), \quad
        \overline{m} \coloneq - u_h - \sup_{\underset{\abs{v} \leq 1}{x \in \R}} L(x, v), 
        \quad
        u_h \coloneq \max_{x \in \R} H(x, m).
    \]
\end{proof}

\begin{proofof}{\ref{item:Stability}}
    With reference to Lemma \ref{lmm:CompactHetero}, fix $(m, M) \in \R^2$ 
    satisfying $m \leq \inf_{\R} \alpha$ and $M \geq \sup_{\R} \alpha$, such that 
    for a.e. $x \in \R$, $m \leq u_o(x) \leq M$. Consider $(v_{j})_j$ and $(w_j)_j$, 
    steady states given by Lemma \ref{lmm:SteadyStateBis}.

    Define the piecewise constant functions 
    \[
        (v_\Delta, w_\Delta) \coloneq \sum_{j \in \Z} (v_{j}, w_j) \1_{I_j},
    \]

    and the constants $\underline{u}$, $\overline{u}$ as they were defined in the proof of 
    Lemma \ref{lmm:SteadyStateBis}.

    The key point is that $v_\Delta$ and $w_\Delta$ are both exact stationary solutions to 
    \[
        \p_t \cU(t, x) + \p_x \bigl(H(\xi_\Delta(x), \cU(t, x)) \bigr) = 0. 
    \]

    Therefore, for all $n \in \N$, recalling the definition of $\cU^n$, 
    \ref{item:StabilityDiscontinuous} of Theorem \ref{th:WellPosednessDiscontinuous} ensures that 
    \[
        v_\Delta \leq u^n \leq w_\Delta
        \implies \forall t \in [t^n, t^{n+1}], \quad 
        v_\Delta \leq \cU^n(t) \leq w_\Delta.
    \]

    Since by construction of $(v_{j})_{j \in \Z}$ and $(w_j)_{j \in \Z}$, we have $v_\Delta \leq \rho^0 \leq w_\Delta$, a direct induction provides:
    \[
        \forall t \in \R^+, \quad v_\Delta \leq u_\Delta(t) \leq w_\Delta,
    \]

    immediately leading to \eqref{eq:StabilityScheme} by definition of $\underline{u}$ and $\overline{u}$. The proof is concluded.
\end{proofof}

\subsection{Convergence of the scheme}
\label{ssec:Convergence}

This section is devoted to the proof of the convergence of the scheme. First, we establish the strong compactness of the approximate solution using the compensated compactness method. Then, in Section \ref{ssec:Convergence}, we prove that the scheme verifies discrete entropy inequalities and pass to the limit using the pre-established compactness.

Below, the constants $\underline{u}, \overline{u}$ refer to the ones from \ref{item:Stability} of Theorem \ref{th:main}.

\subsubsection{Compensated compactness}
\label{sssec:CompComp}

The compensated compactness method and its applications to systems of conservation laws 
is for instance reviewed in \cite{Chen2000, LuBook}. Modified for our use, the 
compensated compactness lemma reads as follows. This section culminates with Corollary \ref{cor:CompComp}, stating the compactness result.

\begin{lemma}
    \label{lmm:CompComp}
    Assume \eqref{eq:smoothness}--\eqref{eq:LocalizedSpace}--\eqref{eq:convexity} hold, 
    ensuring that $H$ is genuinely nonlinear: 
    \begin{equation}
        \label{eq:FluxGNL}
        \forall (x,s) \in \R^2, \quad \text{meas} \bigl(\left\{p \in \R \; : \; 
        \p_u H(x, p) = s \right\} \bigr) = 0.
    \end{equation}

    Let $(u_\eps)_\eps$ be a bounded sequence of $\L{\infty}(\open{0}{+\infty} \times \R, \R)$ such that for all $k \in \R$ and for any $i \in \{1, 2\}$, the sequence $(\p_t S_i (u_\eps) + \p_x Q_i (x, u_\eps))_\eps$ belongs to a compact subset of $\Hloc{-1}(\open{0}{+\infty} \times \R, \R)$, where 
    \begin{equation}
        \label{eq:CompCompEntropies}
        \begin{array}{cc}
            S_1(u) \coloneq u - k & Q_1(x, u) \coloneq H(x, u) - H(x, k) \\[5pt]
            S_2(u) \coloneq H(x, u) - H(x, k) & 
            Q_2(x, u) \coloneq \int_k^u \p_u H(x, \xi)^2 \d{\xi}.
        \end{array} 
    \end{equation}
    
    Then there exists a subsequence of $(u_\eps)_\eps$ that converges in 
    $\Lloc{p}(\open{0}{+\infty} \times \R, \R)$ for all $p \in [1, +\infty \mathclose[$ and a.e. on $\open{0}{+\infty} \times \R$ to some function $u \in \L{\infty}(\open{0}{+\infty} \times \R, \R)$.
\end{lemma}

The main step of the compensated compactness method is to prove that two entropy productions are compact in $\Hloc{-1}(\open{0}{+\infty} \times \R, \R)$. It can be obtained using the following technical result, see \cite{Murat1981, DCP1987, SerreBook}.

\begin{lemma}
    \label{lmm:CompCompTechnical}
    Let $q, r \in \R$ such that $1 < q < 2 < r$. Let $(\mu_\eps)_\eps$ be a sequence of 
    distributions such that:
    
    (i) $(\mu_\eps)_\eps$ belongs to a compact subset of 
    $\Wloc{-1}{q}(\open{0}{+\infty} \times \R, \R)$.

    (ii) $(\mu_\eps)_\eps$ is bounded in $\Wloc{-1}{r}(\open{0}{+\infty} \times \R, \R)$.

    Then $(\mu_\eps)_\eps$ belongs to a compact subset of 
    $\Hloc{-1}(\open{0}{+\infty} \times \R, \R)$.
\end{lemma}

In the remainder of the section, we verify that a generic entropy production of the scheme satisfies the assumptions of Lemma \ref{lmm:CompCompTechnical}. Let us stress that this step is nontrivial: it relies on the properties of the solutions to Riemann problems with discontinuous flux. For the analysis, we take inspiration from \cite{KMRTriangular, KMRWellBalanced}.

Let $(S, Q)$ be an entropy/entropy flux pair associated with $H$. 
The entropy dissipation of $u_\Delta$ is defined as
\begin{equation}
    \label{eq:EntropyDissipation}
    E_\Delta(\phi) \coloneq \int_{0}^{+\infty} \int_{\R} 
    S(u_\Delta) \p_t \phi + Q(\xi_\Delta, u_\Delta)  \p_x \phi \; \d{x} \d{t}, 
    \quad \phi \in \Cc{\infty}(\R^+ \times \R, \R).
\end{equation}

Fix $n \in \N$ and $j \in \Z$. Let us consider the entropy dissipation in 
$P_{j+1/2}^n \coloneq [t^n, t^{n+1}[ \times ]x_j, x_{j+1}[$. By integration by parts 
and the fact that $u_\Delta$ is the exact solution of a Riemann problem in 
$P_{j+1/2}^n$, we write:
\[
    \begin{aligned}
        \iint_{P_{j+1/2}^n} 
        & S(u_\Delta) \p_t \phi + Q(\xi_\Delta, u_\Delta) \p_x \phi \;\d{x} \d{t} \\
        & = \int_{x_j}^{x_{j+1}} S(u_\Delta(t^{n+1}-, x)) \phi(t^{n+1}, x) 
        -  S(u_\Delta(t^{n}, x)) \phi(t^{n}, x) \; \d{x} \\
        & + \int_{t^{n}}^{t^{n+1}} Q(x_{j+1}, u_{j+1}^n) \phi(t, x_{j+1}) 
        - Q(x_{j}, u_{j}^n) \phi(t, x_{j}) \; \d{t} \\
        & + \int_{t^{n}}^{t^{n+1}} [\![ \sigma S(u_\Delta) - 
        Q(\xi_\Delta, u_\Delta) ]\!]_y \; \phi(t, y(t)) \d{t} \\
        & + \int_{t^{n}}^{t^{n+1}} \biggl( Q(x_{j}, u_\Delta(t, x_{j+1/2}-)) 
        - Q(x_{j+1}, u_\Delta(t, x_{j+1/2}+)) \biggr) \phi(t, x_{j+1/2}) \d{t}.
    \end{aligned}
\]

With reference to Section \ref{ssec:DiscontinuousFlux}, in the right-hand side of the 
previous equality:

$\bullet$ The third integral is the contribution of the (eventual) classical shock in 
the solution. Hence, $\sigma$ is the shock velocity, given by the Rankine-Hugoniot 
condition, $y$ is the shock curve and 
\[
    [\![ \sigma S(u_\Delta) - Q(\xi_\Delta, u_\Delta) ]\!]_y
    \coloneq (\sigma S(u_\Delta) - Q(\xi_\Delta, u_\Delta))(t, y(t)+) 
    - (\sigma S(u_\Delta) - Q(\xi_\Delta, u_\Delta))(t, y(t)-).
\]

It is relevant to say that $\xi_\Delta$ is continuous along $\{x=y(t)\}$, equal to 
either $x_j$ or $x_{j+1}$, depending on the sign of $\sigma$. 

$\bullet$ The last integral is the contribution of the (eventual) stationary 
non-classical shock in the solution.

Taking the sum for $n \in \N$ and $j \in \Z$, we see that the entropy dissipation 
rewrites as
\begin{equation}
    \label{eq:EntropyDissipation2}
    \begin{array}{clr}
        E_\Delta(\phi) 
        & \ds{= -\int_{\R} S(u_\Delta(0, x)) \phi(0, x) \d{x}} 
        & \longleftarrow I_1(\phi) \\[5pt]
        & \ds{+ \sum_{n=1}^{+\infty} \int_{\R} 
        (S(u_\Delta(t^{n}-, x)) - S(u_\Delta(t^{n}, x))) 
        \phi(t^n, x) \d{x}} 
        & \longleftarrow I_2(\phi) \\[5pt]
        & \ds{+ \sum_{n=0}^{+\infty} \sum_{j \in \Z} \sum_{y} \int_{t^{n}}^{t^{n+1}} 
        [\![ \sigma S(u_\Delta) - Q(\xi_\Delta, u_\Delta) ]\!]_y \; 
        \phi(t, y(t)) \d{t}} 
        & \longleftarrow I_3(\phi) \\[5pt]
        & \ds{+ \sum_{n=0}^{+\infty} \sum_{j \in \Z} \int_{t^{n}}^{t^{n+1}} 
        \biggl( Q(x_{j}, u_\Delta(t, x_{j+1/2}-)) 
        - Q(x_{j+1}, u_\Delta(t, x_{j+1/2}+)) \biggr) \phi(t, x_{j+1/2}) \; \d{t}} 
        & \longleftarrow I_4(\phi).
    \end{array}
\end{equation}

\bigskip

As we previously explained, for all $n \in \N$ and $j \in \Z$, the sum $\sum_{y}$ in 
\eqref{eq:EntropyDissipation2} is either over an empty set, or over a singleton.
\bigskip

First, let us give a bound on the variation of the approximate solution across the 
discrete time levels.

\begin{lemma}
    \label{lmm:CompCompA}
    Assume \eqref{eq:smoothness}--\eqref{eq:LocalizedSpace}--\eqref{eq:convexity} hold. Let $T > 0$ and $R > 0$. Fix $N, K \in \N^*$ such that $T \in [t^N, t^{N+1} \mathclose[$ and $R \in I_{K}$. Then, there exists $c_1 > 0$ depending on $T$, $R$, $u_o$, and $H$ such that 
    \begin{equation}
        \label{eq:CompCompA}
        \begin{aligned}
            & \sum_{n=1}^{N} \sum_{\abs{j} \leq K} 
        \int_{I_j} \abs{u_\Delta (t^n, x) - u_\Delta(t^n-, x)}^2 \d{x} \leq c_1 \\
        \text{and} \quad & 
        \sum_{n=0}^{N} \sum_{\abs{j} \leq K} \sum_{y} \int_{t^{n}}^{t^{n+1}} 
        [\![ \sigma S(u_\Delta) - Q(\xi_\Delta, u_\Delta) ]\!]_y \; \d{t} \leq c_1.
        \end{aligned}
    \end{equation}
\end{lemma}

\begin{proof}
    For the purpose of this proof, let us choose $S(u) \coloneq \frac{u^2}{2}$ and 
    $Q$ its entropy flux. Let $(\phi_k)_k$ be a sequence of nonnegative 
    test functions that converges to $\phi \coloneq \1_{[0, T] \times [-R, R]}$. At the 
    limit $k \to +\infty$ in \eqref{eq:EntropyDissipation2}, we obtain:
    \[
        I_2(\phi) + I_3(\phi) = \int_{\abs{x} \leq R} \abs{u_\Delta(0, x)}^2 \d{x} - I_4(\phi) 
        \leq 2 R \norm{u_o}_{\L{\infty}(\R)}^2 - I_4(\phi).
    \]  

    Note that $I_3(\phi)$ is nonnegative because the (eventual) shock is classical 
    and therefore produces entropy. 

    Then, notice that by definition of $(u_j^n)_{j, n}$,
    \[
        \begin{aligned}
            I_2(\phi) 
            & = \frac{1}{2} \sum_{n=1}^{N} \sum_{\abs{j} \leq K} \int_{I_j} 
            u_\Delta(t^{n}-, x)^2 - u_\Delta(t^{n}, x)^2 \; \d{x} \\ 
            & = \frac{1}{2} \sum_{n=1}^{N} \sum_{\abs{j} \leq K} \int_{I_j} \biggl\{
            (u_\Delta(t^{n}-, x) - u_\Delta(t^{n}, x))^2
            +2 u_\Delta(t^{n}, x) (u_\Delta(t^{n}-, x)  - u_\Delta(t^{n}, x))
            \biggr\} \d{x} \\
            & = \frac{1}{2} \sum_{n=1}^{N} \sum_{\abs{j} \leq K} \int_{I_j} 
            (u_\Delta(t^{n}-, x) - u_\Delta(t^{n}, x))^2 \d{x}
            + \sum_{n=1}^{N} \sum_{\abs{j} \leq K} u_j \underbrace{\int_{I_j}
            (u_\Delta(t^{n}-, x)  - u_j^n) \d{x}}_{=0}.
        \end{aligned}
    \]

    Finally, as a consequence of Proposition \ref{pp:FluxAJG} \textit{(ii)}, for all 
    $k \in [\underline{u}, \overline{u}]$,
    \[
        \Phi(x_{j+1/2}-, u_\Delta(t, x_{j+1/2}-), k) 
        - \Phi(x_{j+1/2}+, u_\Delta(t, x_{j+1/2}+), k)
        \geq - \cR_{j+1/2}(k, k).
    \]

    Using Remark \ref{rk:ConstantRemainderTerm}, we can bound this remainder term as 
    \[
        \cR_{j+1/2}(k, k) 
        \leq \bigl(2 \norm{\alpha'}_{\L{\infty}(\R)} \sup_{\underset{\underline{u} \leq p \leq \overline{u}}{x \in \R}} \abs{\p_u H(x, p)} 
        + \sup_{\underset{\underline{u} \leq p \leq \overline{u}}{x \in \R}} \abs{\p_x H(x, p)} \bigr) \Delta x.
    \]
    
    Using an approximation argument to pass from Kruzhkov entropies to any entropy, 
    see \cite{KMRTriangular}, we obtain 
    \[
        I_4(\phi) 
        \geq - 2TR \bigl(2 \norm{\alpha'}_{\L{\infty}(\R)} \sup_{\underset{\underline{u} \leq p \leq \overline{u}}{x \in \R}} \abs{\p_u H(x, p)} 
        + \sup_{\underset{\underline{u} \leq p \leq \overline{u}}{x \in \R}} \abs{\p_x H(x, p)} \bigr).
    \]

    Estimate \eqref{eq:CompCompA} follows by putting the bounds on $I_1(\phi)$ and 
    $I_4(\phi)$ together.
\end{proof}

We can convert \eqref{eq:CompCompA} into an estimate on the spatial variation of the 
approximate solutions, following the arguments of  \cite[Page 67]{DiPerna1983}. 
For the sake of clarity, we set 
$\rho_{j+1/2}^{n, \pm} \coloneq u_\Delta(t^n, x_{j+1/2} \pm)$.

\begin{lemma}
    \label{lmm:CompCompB}
    Assume \eqref{eq:smoothness}--\eqref{eq:LocalizedSpace}--\eqref{eq:convexity} hold.
    Let $T > 0$ and $R> 0$. Fix $N, K \in \N^*$ such that 
    $T \in [t^N, t^{N+1} \mathclose[$ and $R\in I_{K}$. Then, there exists $c_2> 0$ depending on $T$, $R$, $u_o$ and $H$ such that 
    \begin{equation}
        \label{eq:CompCompB}
        \sum_{n=0}^{N} \sum_{\abs{j} \leq K} \left\{ 
            (\rho_{j+1/2}^{n, -} - \rho_{j+1/2}^{n, +})^2 
            + (u_{j}^{n} - \rho_{j+1/2}^{n, -})^2 
            + (u_{j+1}^{n} - \rho_{j+1/2}^{n, +})^2 \right\} \Delta x \leq c_2. 
    \end{equation}
\end{lemma}

We are now in position to prove the $\Hloc{-1}$ compactness of the sequence of measures 
that defines the entropy dissipation.

\begin{lemma}
    \label{lmm:CompCompC}
    Assume that $H$ satisfies \eqref{eq:smoothness}--\eqref{eq:LocalizedSpace}--\eqref{eq:convexity}. 
    Let $(S_i, Q_i)_{i \in \{1, 2\}}$ be the entropy/entropy flux pairs defined in 
    Lemma \ref{lmm:CompComp}. Then for any $i \in \{1, 2\}$, the sequence of 
    distributions $\mu_i$ defined by 
    \[
        \mu_i(\phi) \coloneq \int_{0}^{+\infty} \int_{\R} 
        S_i(u_\Delta) \p_t \phi + Q_i(x, u_\Delta)  \p_x \phi \; \d{x} \d{t}, 
        \quad \phi \in \Cc{\infty}(\R^+ \times \R, \R),
    \]
    
    belongs to a compact subset of $\Hloc{-1}(\open{0}{+\infty} \times \R, \R)$.
\end{lemma}

\begin{proof}
    We follow the proofs of \cite[Lemma 5.5]{KMRTriangular} or 
    \cite[Lemma 4.5]{KMRWellBalanced}. To start, let us fix $(S, Q)$ a smooth 
    entropy/entropy flux pair.

    Let us fix a bounded open subset $\Omega \subset \open{0}{+\infty} \times \R$, say 
    $\Omega \subset [0, T] \times [-R, R]$ for some $T > 0$ and $R > 0$. 
    Call $N, K \in \N^*$ such that $T \in [t^N, t^{N+1}[$ and $R \in I_{K}$, and 
    finally, let $\phi \in \Cc{\infty}(\Omega, \R)$. We split $\mu$ as 
    \[
        \mu(\phi) = 
        \int_{0}^{+\infty} \int_{\R} (Q(x, u_\Delta) - Q(\xi_\Delta, u_\Delta)) 
        \p_x \phi \; \d{x} \d{t} + E_\Delta(\phi),
    \]

    where $E_\Delta(\phi)$ is given by \eqref{eq:EntropyDissipation}. 

    The definition of $\xi_\Delta$ and regularity of $H$ ensures that for any 
    $q_1 \in \mathopen]1, 2]$, the first term of the right-hand side is strongly 
    compact in $\W{-1}{q_1}(\Omega, \R)$. We now estimate $E_\Delta(\phi)$.

    First, observe that since $(u_\Delta)_\Delta$ is 
    bounded in $\L{\infty}(\Omega, \R)$, we have 
    \[
        \abs{E_\Delta(\phi)} \leq \bigl(
            \sup_{\underline{u} \leq p \leq \overline{u}} \abs{S(p)} 
        + \sup_{\underset{\underline{u} \leq p \leq \overline{u}}{x \in \R}} \abs{Q(x, p)}
        \bigr) \cdot \norm{\phi}_{\W{1}{\infty}(\Omega)},    
    \]

    implying that $(E_\Delta(\phi))_\Delta$ is bounded in $\Wloc{-1}{r}(\Omega, \R)$ 
    for any $r \in \open{2}{+\infty}$.

    Keeping the notations of \eqref{eq:EntropyDissipation2}, we bound $I_1(\phi)$ and 
    $I_3(\phi)$ as 
    \[
        \abs{I_1(\phi)} \leq \norm{S(u_o)}_{\L{1}([-R, R])} \; \norm{\phi}_{\L{\infty}(\Omega)}, 
        \quad I_3(\phi) \leq c_1 \norm{\phi}_{\L{\infty}(\Omega)},
    \]

    where we used Lemma \ref{lmm:CompCompA} for $I_3(\phi)$. Consequently, 
    $I_1(\phi)$ and $I_3(\phi)$ are bounded in the space $\mathcal{M}(\Omega, \R)$ of 
    bounded Radon measures, which is compactly embedded in $\W{-1}{q_2}(\Omega, \R)$ if 
    $q_2 \in \mathopen]1, 2 \mathclose[$. Therefore, $I_1(\phi)$ and $I_3(\phi)$ belong 
    to a compact subset of $\W{-1}{q_2}(\Omega, \R)$.

    Now, to estimate $I_2(\phi)$, split it as 
    $I_2(\phi) = I_{2,1} (\phi) + I_{2, 2}(\phi)$ where 
    \[
        \begin{aligned}
            I_{2, 1}(\phi) & \coloneq \sum_{n=1}^{N} \sum_{\abs{j} \leq K} \int_{I_j} 
            (S(u_\Delta(t^{n}-, x)) - S(u_j^n)) \phi(t^n, x_j) \d{x} \\
            I_{2, 2}(\phi) & \coloneq \sum_{n=1}^{N} \sum_{\abs{j} \leq K} \int_{I_j} 
            (S(u_\Delta(t^{n}-, x)) - S(u_j^n)) 
            (\phi(t^n, x) - \phi(t^n, x_j)) \d{x}.
        \end{aligned}
    \]

    We handle $I_{2,1}(\phi)$ by introducing an intermediary point $\theta_j^n \in \R$ 
    such that 
    \[
        S(u_\Delta(t^{n}-, x)) - S(u_j^n) 
        = S'(u_j^n) (u_\Delta(t^{n}-, x) - u_j^n) 
        + \frac{S''(\theta_j^n)}{2} (u_\Delta(t^{n}-, x) - u_j^n)^2.
    \]

    Taking into account the definition of $u_j^n$ and Lemma \ref{lmm:CompCompB}, we 
    obtain
    \[
        \abs{I_{2, 1}(\phi)} \leq \frac{c_1}{2} \sup_{\underline{u} \leq p \leq \overline{u}} S''(p) \cdot \norm{\phi}_{\L{\infty}(\Omega)},
    \]

    ensuring that $I_{2, 1}(\phi)$ belong to a compact subset of 
    $\W{-1}{q_2}(\Omega, \R)$.

    We continue by choosing $\alpha \in \mathopen]\frac{1}{2}, 1 \mathclose[$ and then 
    writing:
    \[
        \begin{aligned}
            \abs{I_{2,2}(\phi)} 
            & \leq \norm{\phi}_{\Ck{0, \alpha}(\Omega)} \Delta x^\alpha
            \sum_{n=1}^{N} \sum_{\abs{j} \leq K} \int_{I_j} 
            \abs{S(u_\Delta(t^{n}-, x)) - S(u_j^n)} \d{x}  \\
            & \leq \norm{\phi}_{\Ck{0, \alpha}(\Omega)} \Delta x^{\alpha-1}
            \left\{ \sum_{n=1}^{N} \sum_{\abs{j} \leq K} \left(\int_{I_j} 
            \abs{S(u_\Delta(t^{n}-, x)) - S(u_j^n)} \d{x} \right)^2 \right\}^{1/2}
            \left\{ \sum_{n=1}^{N} \sum_{\abs{j} \leq K} \Delta x^2 \right\}^{1/2} \\
            & \leq \sqrt{\frac{2TR}{\lambda}} \norm{\phi}_{\Ck{0, \alpha}(\Omega)} 
            \Delta x^{\alpha-1/2} \left\{ \sum_{n=1}^{N} \sum_{\abs{j} \leq K} \int_{I_j} 
            \abs{S(u_\Delta(t^{n}-, x)) - S(u_j^n)}^2 \d{x} \right\}^{1/2} \\
            & \leq \sqrt{\frac{2 T R c_1}{\lambda}} 
            \sup_{\underline{u} \leq p \leq \overline{u}} \abs{S'(p)} \cdot 
            \norm{\phi}_{\Ck{0, \alpha}(\Omega)} \; \Delta x^{\alpha-1/2} \qquad
            \text{Lemma \ref{lmm:CompCompA}.}
        \end{aligned}
    \]

    We now take advantage of the compact embedding 
    $\Ck{0, \alpha}(\Omega, \R) \subset \W{1}{p}(\Omega, \R)$, $p=2/(1-\alpha)$ 
    to be sure that $I_{2,2} (\phi)$ belongs to a compact subset of 
    $\W{-1}{q_3}(\Omega, \R)$, 
    $q_3\coloneq 2/(1+\alpha) \in \mathopen]1, \frac{4}{3} \mathclose[$.

    Finally, we estimate $I_4(\phi)$. Let $k \in [\underline{u}, \overline{u}]$ and 
    consider $(S_i, Q_i)_{i \in \{1, 2\}}$ the entropy/entropy flux pairs defined in 
    Lemma \ref{lmm:CompComp}.
    Because of the Rankine-Hugoniot condition,
    \[
        \abs{[\![ Q_{1}(\xi_\Delta, u_\Delta) ]\!]_{x_{j+1/2}}^n}
        = \abs{H(x_j, k) - H(x_{j+1}, k)}, 
    \]

    from which we deduce 
    \[
        \abs{I_{4,1}(\phi)}
        \leq 2TR \; \sup_{x \in \R} \abs{\p_x H(x, k)} 
        \cdot \norm{\phi}_{\L{\infty}(\Omega)},
    \]

    ensuring that $I_{4,1}(\phi)$ belongs to a compact subset of 
    $\W{-1}{q_2}(\Omega, \R)$. On the other hand,
    \[
        \begin{aligned}
            \abs{ [\![ Q_{2}(\xi_\Delta, u_\Delta) ]\!]_{x_{j+1/2}}^n}
            & = \int_{\rho_{j+1/2}^{n,+}}^{\rho_{j+1/2}^{n,-}} 
            \p_u H(x_j, \xi)^2 \d{\xi} 
            + \int_{k}^{\rho_{j+1/2}^{n,+}} \p_u H(x_j, \xi)^2
            - \p_u H(x_{j+1}, \xi)^2 \d{\xi} \\
            & \leq \int_{\rho_{j+1/2}^{n,+}}^{\rho_{j+1/2}^{n,-}} 
            \p_u H(x_j, \xi)^2 \d{\xi} + 2L (\overline{u} - \underline{u}) \sup_{\underset{\underline{u} \leq p \leq \overline{u}}{x \in \R}} 
            \abs{\p_{xu}^2 H(x, p)} \Delta x \\
            & \leq L \underbrace{\int_{\rho^{-}}^{\rho^{+}} 
            \abs{\p_u H(x_j, \xi)} \d{\xi}}_{Q_{2,1}} + 2L (\overline{u} - \underline{u}) \sup_{\underset{\underline{u} \leq p \leq \overline{u}}{x \in \R}} 
            \abs{\p_{xu}^2 H(x, p)} \Delta x,
        \end{aligned}
    \] 

    where $\rho^{-} \coloneq \min\{\rho_{j+1/2}^{n,-}, \rho_{j+1/2}^{n,+}\}$ and 
    $\rho^{+} \coloneq \max\{\rho_{j+1/2}^{n,-}, \rho_{j+1/2}^{n,+}\}$.
    If $u \mapsto \p_u H(x_j, u)$ does not change sign in $\open{\rho^-}{\rho^+}$, then 
    \[
        \begin{aligned}
            \abs{Q_{2, 1}} 
            & = \abs{H(x_j, \rho_{j+1/2}^{n,-}) - H(x_j, \rho_{j+1/2}^{n,+})} \\
            & \leq \underbrace{\abs{H(x_j, \rho_{j+1/2}^{n,-}) - H(x_{j+1}, \rho_{j+1/2}^{n,+})}}_{=0} 
            + \abs{H(x_{j+1}, \rho_{j+1/2}^{n,+}) - H(x_j, \rho_{j+1/2}^{n,+})} \\
            & \leq \sup_{\underset{\underline{u} \leq p \leq \overline{u}}{x \in \R}} 
            \abs{\p_x H(x, p)} \; \Delta x.
        \end{aligned}
    \]

    Otherwise, assume to fix the ideas that $\rho^- \leq \alpha_j \leq \rho^+$. 
    Then, setting 
    $L_H \coloneq \sup_{\underset{\underline{u} \leq p \leq \overline{u}}{x \in 
    \R}} \abs{\p_{uu}^2 H(x, p)}$ we write
    \[
        \begin{aligned}
            Q_{2, 1} & 
            = \int_{\rho^{-}}^{\rho^{+}} \abs{\p_u H(x_j, \xi) - \p_u H(x_j, \alpha_j)} 
            \d{\xi} \\
            & \leq L_H \int_{\rho^{-}}^{\rho^{+}} \abs{\xi - \alpha_j} \d{\xi} \\
            & = \frac{L_H}{2} \left((\rho^+ - \alpha_j)^2 
            + (\rho^- - \alpha_j)^2\right) \\
            & \leq L_H (\rho^+ - \rho^-)^2
            = L_H \; (\rho_{j+1/2}^{n,-} - \rho_{j+1/2}^{n,+})^2.
        \end{aligned}
    \]

    Taking advantage of Lemma \ref{lmm:CompCompB}, we conclude that 
    \[
        \begin{aligned}
            & \sum_{n=0}^{N} \Delta t \sum_{\abs{j} \leq J}  
            \abs{[\![ Q_{2}(\xi_\Delta, u_\Delta) ]\!]_{x_{j+1/2}}^n} \\
            & \leq L \sum_{n=0}^{N} \Delta t \sum_{\abs{j} \leq J}  
            \left\{L_H \; (\rho_{j+1/2}^{n,-} - \rho_{j+1/2}^{n,+})^2 + \sup_{\underset{\underline{u} \leq p \leq \overline{u}}{x \in \R}} \abs{\p_x H(x, p)} \; \Delta x 
            +  2 (\overline{u} - \underline{u}) \sup_{\underset{\underline{u} \leq p \leq \overline{u}}{x \in \R}} 
            \abs{\p_{xu}^2 H(x, p)} \Delta x \right\} \\
            & \leq L \left\{ L_H + \sup_{\underset{\underline{u} \leq p \leq \overline{u}}{x \in \R}} \abs{\p_x H(x, p)} 
            + 2 (\overline{u} - \underline{u}) \sup_{\underset{\underline{u} \leq p \leq \overline{u}}{x \in \R}} 
            \abs{\p_{xu}^2 H(x, p)} \right\} (\lambda c_2 + 2TR),
        \end{aligned}
    \]

    which ensures that $I_{4, 2}(\phi)$ belongs to a compact subset of 
    $\W{-1}{q_2}(\Omega, \R)$. 

    To summarize, for any $i \in \{1, 2\}$, $(\mu_i)_\Delta$ belongs to a compact 
    subset of $\Wloc{-1}{q}(\open{0}{+\infty} \times \R, \R)$, 
    $q \coloneq \min\{q_1, q_2, q_3\} < 2$. 
    Additionally, $(\mu_i)_\Delta$ is bounded in 
    $\Wloc{-1}{r}(\open{0}{+\infty} \times \R, \R)$ for any $r \in \open{2}{+\infty}$. 
    Lemma \ref{lmm:CompCompTechnical} applies to ensure that for any $i \in \{1, 2\}$, 
    $(\mu_i)_\Delta$ belongs to a compact subset of 
    $\Hloc{-1}(\open{0}{+\infty} \times \R, \R)$.
\end{proof}

\begin{corollary}
    \label{cor:CompComp}
    Assume that $H$ satisfies \eqref{eq:smoothness}--\eqref{eq:LocalizedSpace}--\eqref{eq:convexity}. 
    Fix $u_o \in \L{\infty}(\R, \R)$ and let $(u_\Delta)_\Delta$ be the sequence 
    generated by the scheme described in Section \ref{sec:TheScheme}.
    Then, there exists a limit function 
    $u \in \L{\infty}(\open{0}{+\infty} \times \R, \R)$ 
    such that along a subsequence, as $\Delta \to 0$ while satisfying \eqref{eq:CFL}, 
    $(u_\Delta)_\Delta$ converges in $\Lloc{p}(\open{0}{+\infty} \times \R, \R)$ for all $p \in [1, +\infty \mathclose[$ and a.e. on 
    $\open{0}{+\infty} \times \R$ to $u$. 
\end{corollary}

\begin{proof}
    Follows from a combination of 
    Lemma \ref{lmm:CompCompA} and Lemma \ref{lmm:CompCompC}.
\end{proof}

\subsubsection{Convergence}

Equipped with compactness obtained in Corollary \ref{cor:CompComp}, we prove that the limit function is the entropy solution. We will need the following technical result. 

\begin{lemma}{\cite[Lemma 4.8]{KMRWellBalanced}}
    \label{lmm:SignTechnical}
    Let $\Omega \subset \open{0}{+\infty} \times \R$ be a bounded open set, 
    $g \in \L{1}(\Omega, \R)$ and suppose that $(g_\eps)_{\eps}$ converges a.e. on 
    $\Omega$ to $g$. Then, there exists $\cL \subset \R$, at most countable, such that 
    for any $k \in \R \backslash \cL$ 
    \[
        \sgn(g_\eps - k) \limit{\eps}{0} \sgn(g-k) \quad \text{a.e. in} \; \Omega.
    \] 
\end{lemma}

We can now prove our convergence statement.

\begin{proofof}{\ref{item:Convergence}}
    Let us split the proof into several steps. 

    \paragraph{Step 1.} Introduce the piecewise constant function: 
    $\ds{\rho_\Delta \coloneq \sum_{n=0}^{+\infty} \sum_{j \in \Z} 
    u_j^n \1_{[t^n, t^{n+1} \mathclose[ \times I_j}}$.   

    We claim that $\norm{u_\Delta - \rho_\Delta}_{\Lloc{2}(\R^+ \times \R)} \limit{\Delta}{0} 0$. Indeed, 
    for all $(t, x) \in [t^n, t^{n+1} \mathclose[ \times \open{x_{j+1/2}}{x_{j+1}}$,
    \[
        \abs{u_\Delta(t, x) - \rho_\Delta(t, x)}
        = \abs{\cU^n(t, x) - u_{j+1}^n} 
        \leq \abs{\rho_{j+1/2}^{n, +} - u_{j+1}^n},
    \]

    since $\cU^n$ is the solution of a Riemann problem with left state $u_j^n$ and 
    right state $u_{j+1}^n$ at $x = x_{j+1/2}$. Likewise, for all 
    $(t, x) \in [t^n, t^{n+1} \mathclose[ \times \open{x_j}{x_{j+1/2}}$,
    \[
        \abs{u_\Delta(t, x) - \rho_\Delta(t, x)}
        = \abs{\cU^n(t, x) - u_j^n} \leq \abs{\rho_{j+1/2}^{n, -} - u_{j}^n}.
    \]

    Therefore, by Lemma \ref{lmm:CompCompB}, for all $N, K \in \N^*$,
    \[
        \begin{aligned}
            & \sum_{n=0}^N \sum_{\abs{j} \leq K} 
            \int_{t^n}^{t^{n+1}} \int_{x_j}^{x_{j+1}} 
            \abs{u_\Delta(t, x) - \rho_\Delta(t, x)}^2 \d{x} \d{t} \\
            & \leq \frac{1}{2} \sum_{n=0}^N \sum_{\abs{j} \leq K} 
            \left\{ \abs{\rho_{j+1/2}^{n, +} - u_{j+1}^n}^2 
            + \abs{\rho_{j+1/2}^{n, -} - u_{j}^n}^2 \right\} \Delta x \Delta t
            \leq \frac{c_1}{2} \Delta t,
        \end{aligned} 
    \]

    proving the claim. The claim ensures that $(\rho_\Delta)_\Delta$ converges a.e. 
    on $\open{0}{+\infty} \times \R$ to $u$.

    \paragraph{Step 2: Discrete entropy inequalities.} Let 
    $k \in [\underline{u}, \overline{u}]$. Under the CFL condition 
    \eqref{eq:CFL}, we derive the following discrete entropy inequalities, consequence 
    of the monotonicity of the scheme:
    \begin{equation}
        \label{eq:DEI}
        \begin{aligned}
            & \left(\abs{u_j^{n+1} - k} - \abs{u_j^n - k} \right) \Delta x 
            + (\Phi_{j+1/2}^n - \Phi_{j-1/2}^n) \Delta t \\
            & \leq - \sgn(u_j^{n+1} - k) (\fint^{j+1/2}(k, k) - 
            \fint^{j-1/2}(k, k)) \Delta t,
        \end{aligned}
    \end{equation}

    where 
    \[
        \Phi_{j+1/2}^n\coloneq \fint^{j+1/2}(u_j^n \vee k, u_{j+1}^n \vee k) 
        - \fint^{j+1/2}(u_j^n \wedge k, u_{j+1}^n \wedge k).
    \]

    Indeed, on the one hand, by convexity,
    \[
        \begin{aligned}
            \abs{u_j^{n+1} - H_j(k, k, k)} 
            & = \abs{u_j^{n+1} - k + \lambda (\fint^{j+1/2}(k, k) - 
            \fint^{j-1/2}(k, k))} \\
            & \geq \abs{u_j^{n+1} - k} + \lambda \sgn(u_j^{n+1} - k) 
            (\fint^{j+1/2}(k, k) - \fint^{j-1/2}(k, k)).
        \end{aligned}
    \]  

    On the other hand, by monotonicity of the scheme,
    \[
        \begin{aligned}
        \abs{u_j^{n+1} - H_j(k, k, k)}
        & \leq H_j(u_{j-1}^n \vee k, u_j^n \vee k, u_{j+1}^n \vee k) 
        - H_j(u_{j-1}^n \wedge k, u_j^n \wedge k, u_{j+1}^n \wedge k) \\
        & = \abs{u_j^n - k} - \lambda (\Phi_{j+1/2}^n - \Phi_{j-1/2}^n).
        \end{aligned}
    \] 

    Inequality \eqref{eq:DEI} follows by combining these two estimates.

    \paragraph{Step 3: Convergence.} Let $\varphi \in \Cc{\infty}(\R^+\times \R, \R^+)$ 
    and fix $T > 0$, $R$ such that the support of $\varphi$ is included in 
    $[0, T] \times [-R, R]$. Define for all $n \in \N$ and $j \in \Z$, 
    $\varphi_j^n \coloneq \frac{1}{\Delta x} \int_{I_j} \varphi(t^n, x) \d{x}$
    
    With reference to Lemma \ref{lmm:ConsistencyInterfaceFlux}, set 
    $\mu \coloneq \sup_{k \in [\underline{u}, \overline{u}]} \mu_k$ and fix 
    $k \in [\underline{u}, \overline{u}]$.

    Multiply \eqref{eq:DEI} by $\varphi_j^n$ and take the double sum for $n \in \N$ and 
    $j \in \Z$. A summation by parts provides $A + B + C \geq 0$ with
    \[
        \begin{aligned}
            A & \coloneq \sum_{n=1}^{+\infty} \sum_{j \in \Z} 
            \abs{u_j^{n}-k} (\varphi_j^{n} - \varphi_j^{n-1}) \Delta x 
            + \sum_{j \in \Z} \abs{u_j^{o}-k} \varphi_j^{o} \Delta x \\
            B & \coloneq \sum_{n=0}^{+\infty} \sum_{j \in \Z} 
            \Phi_{j+1/2}^n (\varphi_{j+1}^{n} - \varphi_j^n) \Delta t \\
            C & \coloneq -\sum_{n=0}^{+\infty} \sum_{j \in \Z} 
            \sgn(u_j^{n+1} - k) (\fint^{j+1/2}(k, k) - 
            \fint^{j-1/2}(k, k)) \varphi_j^n \Delta t.
        \end{aligned}
    \]

    Clearly,
    \[
        A
        \limit{\Delta}{0} 
        \int_{0}^{+\infty} \int_{\R} \abs{u-k} \p_t \varphi \; \d{x}\d{t} 
        + \int_{\R} \abs{u_o(x)-k} \varphi(0, x) \d{x}.
    \]

    Consider now $B$. For all $n \in \N$ and $j \in \Z$, write 
    \[
        \begin{aligned}
            \Phi_{j+1/2}^n
            & = \underbrace{\fint^{j+1/2}(u_j^n \vee k, u_{j+1}^n \vee k) 
            - \fint^{j+1/2}(u_j^n \vee k, u_{j}^n \vee k)}_{B_1} \\
            & + \underbrace{\fint^{j+1/2}(u_j^n \vee k, u_{j}^n \vee k) 
            - H(x_{j+1}, u_j^n \vee k)}_{B_{2}} + 
            \underbrace{\Phi(x_{j+1}, u_j^n, k) - \Phi(x_{j}, u_j^n, k)}_{B_3}
            + \Phi(x_{j}, u_j^n, k) \\
            & + \underbrace{H(x_{j+1}, u_j^n \wedge k) - 
            \fint^{j+1/2}(u_j^n \wedge k, u_{j}^n \wedge k)}_{B_4} 
            + \underbrace{\fint^{j+1/2}(u_j^n \wedge k, u_{j}^n \wedge k) 
            - \fint^{j+1/2}(u_j^n \wedge k, u_{j+1}^n \wedge k)}_{B_5}.
        \end{aligned}
    \]

    We see that
    \[
        \begin{aligned}
        \abs{\sum_{n=0}^{+\infty} \sum_{j \in \Z} (B_1 + B_5) 
        (\varphi_{j+1}^{n} - \varphi_j^n) \Delta t} & \leq  
        2 L \bigl( T \norm{\p_{xx}^2 \varphi}_{\L{\infty}(\R^+, \L{1}(\R))} 
        + \norm{\p_{tx}^2 \varphi}_{\L{1}(\R^+ \times \R)} \bigr) (\Delta x + \Delta t) \\
        & + 2 L 
        \int_{0}^{+\infty} \int_{\R} \abs{\rho_\Delta(t, x+\Delta x) - \rho_\Delta(t, x)} 
        \cdot \abs{\p_x \varphi(t, x)} \d{x} \d{t}.
        \end{aligned}
    \]

    Since $(\rho_\Delta)_\Delta$ converges a.e. on $\open{0}{+\infty} \times \R$, 
    $(\rho_\Delta \p_x \varphi)_\Delta$ is strongly compact in 
    $\L{1}(\open{0}{+\infty} \times \R, \R)$. As a consequence of the 
    Riesz-Fréchet-Kolmogorov compactness characterization, see 
    \cite[Chapter 4, Section 5]{BrezisBook}, 
    \[
        \int_{0}^{+\infty}\int_{\R} \abs{\rho_\Delta(t, x+\Delta x)-\rho_\Delta(t, x)} 
        \cdot \abs{\p_x \varphi(t, x)} \d{x} \d{t} \limit{\Delta}{0} 0.
    \]

    On the other hand, using Lemma \ref{lmm:ConsistencyInterfaceFlux}, we can estimate 
    $B_2+B_4$ as:
    \[ 
        \abs{\sum_{n=0}^{+\infty} \sum_{j \in \Z} (B_2 + B_4) 
        (\varphi_{j+1}^n - \varphi_j^n) \Delta t}
        \leq 2 \mu T \cdot \norm{\p_{x} \varphi}_{\L{\infty}(\R^+, \L{1}(\R))} \Delta x.
    \] 

    Regarding $B_3$, it is straightforward that 
    \[
        \abs{\sum_{n=0}^{+\infty} \sum_{j \in \Z} B_3
        (\varphi_{j+1}^n - \varphi_j^n) \Delta t} 
        \leq T \; \sup_{\underset{p, k \in [\underline{u}, \overline{u}]}{x \in \R}} 
        \abs{\p_x \Phi(x, p, k)} 
        \cdot \norm{\p_{x} \varphi}_{\L{\infty}(\R^+, \L{1}(\R))} \Delta x.
    \]

    We now handle the last term in $B$. We have:
    \[
        \begin{aligned}
            & \sum_{n=0}^{+\infty} \sum_{j \in \Z} \Phi(x_{j}, u_j^n, k) (\varphi_{j+1}^n - \varphi_j^n) \Delta t \\
            & = \Delta t \sum_{n=0}^{+\infty} \sum_{j \in \Z} \int_{I_j}
            \Phi(x_j, u_j^n, k) 
            \left( \frac{\varphi(t^n, x+ \Delta x) - \varphi(t^n, x)}{\Delta x} - 
            \p_x \varphi(t^n, x) \right) \d{x} \\
            & + \sum_{n=0}^{+\infty} \sum_{j \in \Z} \int_{t^n}^{t^{n+1}} \int_{I_j} \Phi(x_j, u_j^n, k) (\p_x \varphi(t^n, x) - \p_x \varphi(t, x)) \d{x} \d{t} \\
            & + \sum_{n=0}^{+\infty} \sum_{j \in \Z} \int_{t^n}^{t^{n+1}} \int_{I_j} (\Phi(x_j, u_j^n, k) - \Phi(x, u_j^n, k)) \; 
            \p_x \varphi(t, x) \d{x} \d{t}
            + \int_{0}^{+\infty} \int_{\R} \Phi(x, \rho_\Delta, k) \p_x \varphi(t, x) \d{x} \d{t} \\
            & \leq T \sup_{\underset{\underline{u} \leq p \leq \overline{u}}{x \in \R}} \abs{\Phi(x, p, k)} \cdot \norm{\p_{tx}^2 \varphi}_{\L{\infty}(\R^+, \L{1}(\R))} \Delta x + 
            \sup_{\underset{\underline{u} \leq p \leq \overline{u}}{x \in \R}} 
            \abs{\Phi(x, p, k)} \cdot 
            \norm{\p_{tx}^2 \varphi}_{\L{1}(\R^+ \times \R)} \Delta t \\
            & + \sup_{\underset{\underline{u} \leq p \leq \overline{u}}{x \in \R}} |\p_x \Phi(x, p, k)| \cdot \norm{\p_{x}\varphi}_{\L{1}(\R^+ \times \R)} \Delta x 
            + \int_{0}^{+\infty} \int_{\R} 
            \Phi(x, \rho_\Delta, k) \p_x \varphi(t, x) \d{x} \d{t},
        \end{aligned}
    \]
    
    ensuring that 
    \[
        B \limit{\Delta}{0} 
        \int_{0}^{+\infty} \int_{\R} 
        \Phi(x, u(t, x), k) \p_x \varphi(t, x) \; \d{x} \d{t}.
    \]

    Finally, introduce $F_\Delta$ the continuous, piecewise linear affine function 
    such that $F_\Delta(x_{j+1/2}) = \fint^{j+1/2}(k, k)$ for all $j \in \Z$. Let us 
    recall that because of \eqref{eq:LocalizedSpace}, $F_\Delta$ is constant outside a 
    compact subset.    
    
    As a consequence of Lemma \ref{lmm:ConsistencyInterfaceFlux}, $(F_\Delta)_\Delta$ 
    converges uniformly to $x \mapsto H(x, k)$ on $\R$. Moreover, since 
    $(F_\Delta)_\Delta$ is bounded in $\Lip(\R, \R)$, $(F_\Delta')_\Delta$ converges 
    weakly* in $\L{\infty}(\R, \R)$, and the limit can only be $x\mapsto \p_x H(x, k)$. 
    We can now conclude with $C$. Write $C = C_1 + C_2$, where 
    \[
        \begin{aligned}
            C_1 & \coloneq 
            -\sum_{n=1}^{+\infty} \sum_{j \in \Z} 
            \sgn(u_j^{n} - k) (\fint^{j+1/2}(k, k) - 
            \fint^{j-1/2}(k, k)) (\varphi_j^{n-1} - \varphi_j^{n}) \Delta t \\
            C_2 & \coloneq 
            -\sum_{n=1}^{+\infty} \sum_{j \in \Z} \int_{I_j}
            \sgn(u_j^{n} - k) F_\Delta' (x) \varphi_j^n \Delta t.
        \end{aligned}
    \]

    On the one hand, using Lemma \ref{lmm:ConsistencyInterfaceFlux}, we have 
    \[
        \begin{aligned}
            \abs{C_1} 
            \leq 2 T R \left(2 \mu + \sup_{x \in \R} \abs{\p_x H(x, k)} \right) 
            \norm{\p_t \varphi}_{\L{\infty}(\R^+ \times \R)} \Delta t.
        \end{aligned}
    \]

    On the other hand, we take advantage of Lemma \ref{lmm:SignTechnical} to obtain that
    \[
        \lim_{\Delta \to 0} C_2 = - \int_{0}^{+\infty} \int_{\R} \sgn(u -k) \;
        \p_x H(x, k) \varphi(t, x) \; \d{x} \d{t},
    \]

    for any $k \in [\underline{u},\overline{u}]\backslash \cL$. Then, repeat the 
    argument from \cite{KRT2004} to extend it for all 
    $k\in[\underline{u},\overline{u}]$.

    By passing to the limit $\Delta \to 0$ in $A + B + C \geq 0$, we proved that 
    \eqref{eq:ES} holds, concluding the proof that $u$ is the 
    entropy solution to \eqref{eq:CL}.
\end{proofof}

\begin{thx}
    The author thanks Nicolas Seguin for enlightening discussions.

    The author acknowledges the PRIN 2022 project \emph{Modeling, Control and
    Games through Partial Differential Equations} (D53D23005620006),
    funded by the European Union - Next Generation EU.
\end{thx}


{\small
  \bibliography{heterogeneous_fvs}
  \bibliographystyle{abbrv}
}

\end{document}